\documentclass[dvips, preprint]{imsart}

\usepackage{amsmath,amsthm,amssymb,amscd,graphicx,color}
\usepackage{epsfig,amsfonts,latexsym,natbib, dsfont}
\usepackage{verbatim}
\usepackage{hyperref}
\hypersetup{
 bookmarksopen=true,
citebordercolor=0 1 0,
 linkbordercolor=1 0 0,
  menubordercolor=1 0 0,
 }
\usepackage{cleveref}
\usepackage{color}

\startlocaldefs

 \newcommand{\bbn}{\mathbb{N}}
 \newcommand{\bbz}{\mathbb{Z}}
 \newcommand{\bbr}{\mathbb{R}}

 \newcommand{\1}{\mathds{1}}
 \newcommand{\e}{\mbox{e}}
 
\newcommand{\norm}{\gamma}
\newcommand{\ka}{\kappa}
\newcommand{\Lamb}{\Lambda}

 \newcommand{\weak}{\Rightarrow}

 \newcommand{\bbE}{{\mathbb{E}}}

 \newcommand{\bbP}{{\mathbb{P}}}

 \newcommand{\wt}{\widetilde}

\newcommand{\tcr}{\color{red}}

\newtheorem{Theorem}{Theorem}[section]
\crefname{Theorem}{Theorem}{Theorems}
\Crefname{Theorem}{Theorem}{Theorems}

\crefname{Corollary}{Corollary}{Corollaries}
\Crefname{Corollary}{Corollary}{Corollaries}

\crefname{Proposition}{Proposition}{Propositions}
\Crefname{Proposition}{Proposition}{Propositions}
\newtheorem{Lemma}[Theorem]{Lemma}
\crefname{Lemma}{Lemma}{Lemmata}
\Crefname{Lemma}{Lemma}{Lemmata}

\crefname{Assumption}{Assumption}{Assumptions}
\Crefname{Assumption}{Assumption}{Assumptions}

\theoremstyle{definition}
\crefname{Definition}{Definition}{Definitions}
\Crefname{Definition}{Definition}{Definitions}
\newtheorem{Remark}[Theorem]{Remark}
\crefname{Remark}{Remark}{Remarks}
\Crefname{Remark}{Remark}{Remarks}

\crefname{Example}{Example}{Examples}
\Crefname{Example}{Example}{Examples}

\numberwithin{equation}{section}

\numberwithin{equation}{section}
\newcommand{\eqdef}{\stackrel{d}{=}}

\newcommand{\bX}{\mathbf{X}}
\newcommand{\Zd}{\mathbb{Z}^d}
\newcommand{\SaS}{S \alpha S}
\newcommand{\mC}{\mathcal{C}}
\newcommand{\mD}{\mathcal{D}}

\newcommand{\Leb}{\mbox{Leb}}
\newcommand{\PRM}{\mbox{PRM}}

\newtheorem{remark}[Theorem]{Remark}
\newtheorem{propn}[Theorem]{Proposition}

\newtheorem{example}[Theorem]{Example}

\newtheorem{lemma}[Theorem]{Lemma}
\newtheorem{cor}[Theorem]{Corollary}

\endlocaldefs

\allowdisplaybreaks

\begin{document}

\begin{frontmatter}

\title{Stable Random Fields, Point Processes and Large Deviations}
\runtitle{Stable Fields and Large Deviations}

\author{\fnms{Vicky} \snm{Fasen$^\ast$}\ead[label=e2]{vicky.fasen@kit.edu}}
\address{Institute of Stochastics, Karlsruhe Institute of Technology, D-76133 Karlsruhe, Germany.\\vicky.fasen@kit.edu}\thankstext{t1}{Vicky Fasen's work was partly done when she was at RiskLab, ETH Zurich. Financial support by the Deutsche Forschungsgemeinschaft through the research grant FA 809/2-2 is gratefully acknowledged.
}
\and
\author{\fnms{Parthanil} \snm{Roy$^\dagger$}\ead[label=e1]{parthanil@isical.ac.in}}
\thankstext{t2}{Parthanil Roy's research was supported by Cumulative
Professional Development Allowance from Ministry of Human Resource Development, Government of India and the project RARE-318984 (a Marie Curie FP7 IRSES Fellowship) at Indian Statistical Institute Kolkata.}
\address{Statistics and Mathematics Unit, Indian Statistical Institute, Kolkata 700108, India.\\parthanil@isical.ac.in}


\runauthor{V. Fasen and P. Roy}

\begin{abstract}
We investigate the large deviation behaviour of a point process sequence
based on a stationary symmetric
$\alpha$-stable ($0<\alpha<2$) discrete-parameter random field using the
framework of \cite{hult:samorodnitsky:2010}. Depending on the ergodic theoretic and group theoretic structures of the underlying nonsingular group action, we observe different large deviation behaviours of this point process sequence. We use our results to study the large deviations of various functionals (e.g., partial sum, maxima, etc.) of stationary symmetric stable fields.
\end{abstract}



\end{frontmatter}

\section{Introduction}
In this paper, we investigate the large deviation behaviours of point processes and partial sums of
stationary \emph{symmetric $\alpha$-stable} ($S\alpha S$) random fields with $\alpha\in(0,2)$.
A random field $\mathbf{X}:=\{X_t\}_{t \in \mathbb{Z}^d}$ is called a \textit{stationary symmetric $\alpha$-stable} discrete-parameter random
field if for all $k \geq 1$, for all $s, t_1, t_2,\ldots, t_k \in \mathbb{Z}^d$, and for all $c_1, c_2, \ldots, c_k \in \mathbb{R}$,
$
\sum_{i=1}^k c_i X_{t_i+s}
$
follows an $S \alpha S$ distribution that does not depend on $s$. See, for example, \cite{samorodnitsky:taqqu:1994} for detailed descriptions on $S \alpha S$ distributions and processes.

The study of rare events and large deviations for heavy-tailed distributions and processes has been of considerable importance starting from the classical works of \cite{heyde:1967a, heyde:1967b, heyde:1968}, \cite{nagaev:1969b,nagaev:1969a}, \cite{nagaev:1979}; see also the technical report of \cite{cline:hsing:1991}. Some of the more recent works in this area include \cite{mikosch:samorodnitsky:2000a}, \cite{rachev:samorodnitsky:2001}, \cite{hult:lindskog:mikosch:samorodnitsky:2005}, \cite{denisov:dieker:shneer:2008}, \cite{hult:samorodnitsky:2010}, etc. When studying the probability of rare events, it is usually important not only to determine the size and the frequency of clusters of extreme values but also to capture the intricate structure of the clusters. For this reason, \linebreak \cite{hult:samorodnitsky:2010} developed a theory to study large deviation behaviors at the level of point processes
to get a better grasp on how rare events occur. Their work relies on convergence of measures that was introduced in \cite{hult:lindskog:2006}. See also the recent works of \cite{das:mitra:resnick:2013} and \cite{Lindskog:Resnick:Roy}, which extended this convergence to more general situations.


Inspired by the works of \cite{davis:resnick:1985} and \cite{davis:hsing:1995}, \cite{resnick:samorodnitsky:2004} studied the asymptotic behaviour of a point process sequence induced by a stationary symmetric stable process. This work was extended to stable random fields by \cite{roy:2010a}. In the present work, we take a slightly stronger version of the point process sequence considered in \cite{roy:2010a} and use the framework introduced by \cite{hult:samorodnitsky:2010} to investigate the corresponding large deviation behaviour. We observe that this point process large deviation principle depends on the ergodic theoretic and group theoretic properties of the underlying nonsingular $\bbz^d$-action through the works of \cite{rosinski:1995, rosinski:2000} and \cite{roy:samorodnitsky:2008}. Just as in \cite{samorodnitsky:2004a, samorodnitsky:2004b} (see also \cite{roy:2010b}), we notice a phase transition that can be regarded as a passage from shorter to longer memory.

The paper is organized as follows. In Section~\ref{Section:Preliminaries}, we present  background on ergodic theory of nonsingular group actions and integral representations of $S\alpha S$ random fields, and describe a special type of convergence of measures. The large deviation behaviors of the associated point processes are considered separately
for stationary $S\alpha S$ random fields generated by dissipative group actions (reflecting shorter memory) in Section~\ref{section:dissipative}, and generated by conservative group actions (reflecting longer memory) in Section~\ref{section:conservative}. Finally, in Section~\ref{section:classical:large deviation}, we obtain the large deviation principle for the partial sum sequence of a stationary $S\alpha S$ random field using continuous mapping theorem.

We introduce some notations that we are going to use throughout this paper. For two sequences of real numbers $\{a_n\}_{n\in\mathbb{N}}$ and $\{b_n\}_{n\in\mathbb{N}}$ the
notation $a_n \sim b_n$ means $a_n/b_n \to 1$ as $n \to \infty$. For $u, v \in \mathbb{Z}^d$, $u = (u_{1}, u_{2}, \ldots, u_{d}) \leq v =(v_{1},
v_{2}, \ldots, v_{d})$ means $u_{i} \leq v_{i}$ for all
$i=1,2,\ldots,d$; $[u, v]$ is the set \linebreak $\{t \in \mathbb{Z}^d: u \leq t \leq v\}$; $\|u\|_\infty:=\max_{1 \leq i \leq d} \,|u_{i}|$ and $\mathbf{0}_d=(0,0,\ldots,0)$, \linebreak $\mathbf{1}_d=(1,1,\ldots,1)$ are elements of $\mathbb{Z}^d$.
For $x\in\bbr$ we define $x^+:=\max(x,0)$ and $x^-:=\max(-x,0)$. Weak convergence is denoted by $\weak$. For some standard Borel space $(S,\mathcal{S})$ with
$\sigma$-finite measure $\mu$ we define the space $L^{\alpha}(S,\mu):=\left\{f:S\to\mathbb{R} \mbox{ measurable}: \|f\|_\alpha <\infty \right\}$ with $\|f\|_\alpha:=\left(\int_S|f(s)|^{\alpha}\,\mu(ds)\right)^{1/\alpha}$.
For two random variables $Y$, $Z$ (not necessarily defined on the same probability space), we write
$Y\stackrel{\text{d}}{=}Z$ if $Y$ and $Z$ are identically distributed. For two random fields $\{Y_t\}_{t \in \mathbb{Z}^d}$ and $\{Z_t\}_{t \in \mathbb{Z}^d}$, the notation $Y_t\stackrel{\text{d}}{=}Z_t$, $t \in \mathbb{Z}^d$ means that they have same finite-dimensional distributions.

\section{Preliminaries} \label{Section:Preliminaries}

In this section, we present the mathematical background on (a) nonsingular group actions, (b) stationary symmetric $\alpha$-stable random fields and (c) Hult-Lindskog-Samorodnitsky (HLS) convergence. The connection between the first two topics will be clear in this section and the third one will be useful in the entire paper.

\subsection{Nonsingular group actions}

Suppose $(G, +)$ is a countable Abelian group with identity element $\textbf{\e}$ and $(S,\mathcal{S},\mu)$ is a $\sigma$-finite standard Borel space. A collection $\{\phi_t\}_{t \in G}$ of measurable maps of $S$ into itself is called a \emph{nonsingular $G$-action} if $\phi_\textbf{\e}$ is the identity map on $S$, $\phi_{t_1+t_2}=\phi_{t_1} \circ \phi_{t_2}$ for all $t_1, t_2 \in G$ and each $\mu \circ \phi_t^{-1}$ is an equivalent measure of $\mu$; see \cite{aaronson:1997}, \cite{krengel:1985} and \cite{zimmer:1984}. Nonsingular actions are also known as {\em quasi-invariant actions}  in the literature (see \cite{varadarajan:1970}). A collection of measurable $\pm 1$-valued maps $\{c_t\}_{t\in G}$ defined on $S$ is called a (measurable) \emph{cocycle} for $\{\phi_t\}_{t \in G}$ if for all $t_1,t_2 \in G$, $c_{t_1+t_2}(s)=c_{t_2}(s) c_{t_1}\big(\phi_{t_2}(s)\big)$ for all $s \in S$.

A measurable set $W \subseteq S$ is called a \emph{wandering set} for the nonsingular $G$-action
$\{\phi_t\}_{t \in G}$ if $\{\phi_t(W):\;t\in G\}$ is a pairwise
disjoint collection. The set $S$ can be decomposed into two
disjoint and invariant parts as follows: $S=\mC \cup \mD$ where $\mathcal{D} = \bigcup_{t \in G} \phi_t(W^\ast)$ for some wandering set $W^\ast \subseteq S$, and $\mathcal{C}$ has no wandering subset of positive $\mu$-measure; see \cite{aaronson:1997} and \cite{krengel:1985}. This decomposition is called the {\em Hopf decomposition}, and the sets $\mathcal{C}$ and $\mathcal{D}$ are called {\em conservative} and {\em dissipative} parts (of $\{\phi_t\}_{t \in G}$), respectively. The action is called conservative if $S=\mathcal{C}$ and
dissipative if $S=\mathcal{D}$.

\subsection{Stationary symmetric stable random fields}

Every stationary $\SaS$ random field $\mathbf{X}$ admits an integral representation of the form
\begin{eqnarray}
X_t\eqdef\int_S c_t(s){\left(\frac{d \mu \circ \phi_t}{d
\mu}(s)\right)}^{1/\alpha}f \circ \phi_t(s) M(ds),\;\; t \in
\mathbb{Z}^d\,, \label{repn_integral_stationary}
\end{eqnarray}
where $M$ is an $S \alpha S$ random measure on some standard Borel space $(S,\mathcal{S})$ with $\sigma$-finite control measure $\mu$, $f \in L^{\alpha}(S,\mu)$, $\{\phi_t\}_{t \in \mathbb{Z}^d}$
is a nonsingular $\mathbb{Z}^d$-action on $(S, \mathcal{S},\mu)$, and $\{c_t\}_{t
\in \mathbb{Z}^d}$ is a measurable cocycle for $\{\phi_t\}$; see \cite{rosinski:1995, rosinski:2000}. We say that a stationary $S\alpha S$ random field
$\{X_t\}_{t\in \mathbb{Z}^d}$ is generated by a nonsingular
$\mathbb{Z}^d$-action $\{\phi_t\}$ on $(S,\mathcal{S}, \mu)$ if it has an
integral representation of the form $(\ref{repn_integral_stationary})$ satisfying
the full support condition
$
\bigcup_{t \in
\mathbb{Z}^d} support( f \circ \phi_t)=S,
$
which can be assumed without loss of generality.

The Hopf decomposition of $\{\phi_t\}_{t \in \Zd}$ induces the following unique (in law) decomposition of the random field $\bX$,
\begin{equation*}
X_t \eqdef  \int_{\mC} f_t(s)M(ds)+\int_{\mD} f_t(s)M(ds)=:X^{\mC}_t+X^{\mD}_t,\;\; t \in \mathbb{Z}^d,
\label{decomp_of_X_t}
\end{equation*}
into a sum of two independent random fields $\bX^\mathcal{C}$ and $\bX^\mathcal{D}$ generated by a conservative and a dissipative $\Zd$-action, respectively; see \cite{rosinski:1995, rosinski:2000}, and \cite{roy:samorodnitsky:2008}. This decomposition reduces the study of stationary $S \alpha S$ random fields to that of the ones generated by conservative and dissipative actions.

It was argued by \cite{samorodnitsky:2004a} (see also \cite{roy:samorodnitsky:2008}) that stationary $S\alpha S$ random fields generated by conservative actions have longer
memory than those  generated by dissipative actions and therefore, the following dichotomy can be observed:
\begin{equation*}
n^{-d/\alpha} \max_{\|t\|_\infty \leq n}|X_t| \Rightarrow \left\{
                                     \begin{array}{ll}
                                     c_\bX \xi_\alpha & \mbox{ if $\bX$ is generated by a dissipative action,} \\
                                     0              & \mbox{ if $\bX$ is generated by a conservative action}
                                     \end{array}
                              \right. 
\end{equation*}
as $n \rightarrow \infty$. Here $\xi_\alpha$ is a standard Frech\'{e}t type extreme value random variable with distribution function
\begin{equation}
\bbP(\xi_\alpha \leq x)=\e^{-x^{-\alpha}},\;\,x > 0,  \label{cdf_of_Z_alpha}
\end{equation}
and $c_\bX$ is a positive constant depending on the random field $\bX$. In the present work, we observe a similar phase transition in the large deviation principles of the point processes, partial sums, order statistics, etc.~as we pass from dissipative to conservative $\mathbb{Z}^d$-actions in the integral representation \eqref{repn_integral_stationary}.

\subsection{The Hult-Lindskog-Samorodnitsky convergence} \label{subsec:HLS_conv}

Fix a nonnegative integer $q$. Let $\mathbb{M}^q$ be the space of all Radon measures on
$$
\mathbb{E}^q:= [-1,1]^d \times \big([-\infty, \infty]^{[-q\mathbf{1}_d, q\mathbf{1}_d]} \setminus\{0\}^{[-q\mathbf{1}_d, q\mathbf{1}_d]}\big)
$$
equipped with the vague topology. Note that $\mathbb{E}^q$ is a locally compact, complete and separable metric space. Therefore, $C^+_K(\mathbb{E}^q)$, the space of all non-negative real-valued continuous functions defined on $\mathbb{E}^q$ with compact support, admits a countable dense subset consisting only of Lipschitz functions; see \cite{kallenberg:1983} and \cite{resnick:1987}.

Using the above mentioned countable dense subset, $\mathbb{M}^q$ can be identified with a closed subspace of $\left[0,\infty\right)^\infty$ in parallel to \cite{hult:samorodnitsky:2010}, p.~36.  In particular, it transpires that $\mathbb{M}^q$ is also a complete and separable metric space under the vague metric (see \cite{resnick:1987}, Proposition~3.17). Let $\mathbf{M}_0(\mathbb{M}^q)$ denote the space of all Borel measures $\rho$ on $\mathbb{M}^q$ satisfying $\rho(\mathbb{M}^q \setminus B(\O,\varepsilon))<\infty$ for all $\varepsilon>0$ (here $B(\O,\varepsilon)$ is the open ball of radius $\varepsilon$ around the null measure $\O$ in the vague metric). Define the \emph{Hult-Lindskog-Samorodnitsky} (HLS) convergence $\rho_n \to \rho$ in $\mathbf{M}_0(\mathbb{M}^q)$ by $\rho_n(f) \to \rho(f)$ for all $f \in C_{b,0}(\mathbb{M}^q)$,
the space of all bounded continuous functions on $\mathbb{M}^q$ that vanish in a neighbourhood of $\O$; see Theorem~2.1 in \cite{hult:lindskog:2006} and
 Theorem~2.1 in \cite{Lindskog:Resnick:Roy}. This set up is the same as in \cite{hult:samorodnitsky:2010} except that the space $\mathbb{M}^q$ includes all Radon measures in $\mathbb{E}^q$, not just the Radon point measures.

 Observe that the space $\mathbb{M}_p^q$ of Radon point measures on $\mathbb{E}^q$ is a closed subset of $\mathbb{M}^q$ (see \cite{resnick:1987}, Proposition~3.14) and hence a complete and separable metric space under the vague metric (see \cite{resnick:1987}, Proposition~3.17). The space $\mathbf{M}_0(\mathbb{M}_p^q)$ (and the HLS convergence therein) can be defined in the exact same fashion; see \cite{hult:samorodnitsky:2010}, pp.~36. In fact, $\mathbf{M}_0(\mathbb{M}_p^q)$ can be viewed as a subset of $\mathbf{M}_0(\mathbb{M}^q)$ using the following natural identification: $\rho \in \mathbf{M}_0(\mathbb{M}_p^q)$ needs to be identified with its extension to $\mathbb{M}^q$ that puts zero measure on $\mathbb{M}^q \setminus \mathbb{M}_p^q$.

For all $g_1,\,g_2 \in C^+_K(\mathbb{E}^q)$ and for all $\epsilon_1,\,\epsilon_2>0$, define a function \linebreak $F_{g_1,g_2,\epsilon_1,\epsilon_2}: \mathbb{M}^q\rightarrow \left[0,\infty\right)$ by
\begin{eqnarray} \label{eq:F}
F_{g_1,g_2,\epsilon_1,\epsilon_2}(\xi):=\left(1-\e^{-(\xi(g_1)-\epsilon_1)_{+}}\right)\left(1-\e^{-(\xi(g_2)-\epsilon_2)_{+}}\right), \;\;\xi\in\mathbb{M}^q.
\end{eqnarray}
Define, for any $\rho \in \mathbf{M}_0(\mathbb{M}^q)$, for all $g_1,\,g_2 \in C^+_K(\mathbb{E}^q)$ and for all $\epsilon_1,\,\epsilon_2>0$,
\begin{eqnarray*}
\rho(F_{g_1,g_2,\epsilon_1,\epsilon_2}):=\int_{\mathbb{M}^q}F_{g_1,g_2,\epsilon_1,\epsilon_2}(\xi) d\rho(\xi).
\end{eqnarray*}
Following verbatim the arguments in the appendix  of \cite{hult:samorodnitsky:2010}  (more specifically, Theorem A.2), the following result can be established.

\begin{propn} \label{propn:suff:condn:HLS:conv} Let $\rho,\rho_1,\rho_2,\ldots$ be in $\mathbf{M}_0(\mathbb{M}^q)$ and $$\rho_n(F_{g_1,g_2,\epsilon_1,\epsilon_2}) \to \rho(F_{g_1,g_2,\epsilon_1,\epsilon_2}) \quad \mbox{ as }n \to \infty$$  for all Lipschitz $g_1,\,g_2 \in C^+_K(\mathbb{E}^q)$ and for all $\epsilon_1,\,\epsilon_2>0$. Then the HLS convergence $\rho_n \to \rho$ holds in $\mathbf{M}_0(\mathbb{M}^q)$.
\end{propn}

\section{The dissipative case} \label{section:dissipative}

Suppose $\bX:=\{X_t\}_{t \in \mathbb{Z}^d}$ is a stationary $S \alpha S$ random field generated by a dissipative group action. In this case, it has been established by \cite{rosinski:1995, rosinski:2000} and \cite{roy:samorodnitsky:2008} that $\bX$ is a stationary {\em mixed moving average random field} (in the sense of \cite{surgailis:rosinski:mandrekar:cambanis:1993}). This means that $\bX$ has the integral representation
\begin{eqnarray}
X_t \eqdef \int_{W \times
{\mathbb{Z}^d}}f(v,u-t)\,M(dv,du),\;\;\;t \in {\mathbb{Z}^d}\,,
\label{repn_mixed_moving_avg}
\end{eqnarray}
where $f \in L^{\alpha}(W \times {\mathbb{Z}^d}, \nu \otimes \zeta)$, $\nu$ is a
$\sigma$-finite measure on a standard Borel space $(W,
\mathcal{W})$, $\zeta$ is the counting measure on $\mathbb{Z}^d$, and $M$ is a $\SaS$ random measure on $W \times {\mathbb{Z}^d}$ with control measure $\nu\otimes \zeta$
(cf. \cite{samorodnitsky:taqqu:1994}).

Suppose $\nu_\alpha$ is the
symmetric measure on $[-\infty,\infty] \setminus \{0\}$ given by
\begin{equation}
\nu_\alpha\left(x,\infty\right]=\nu_\alpha\left[-\infty,-x\right)=x^{-\alpha},\;\;x>0\,. \label{defn:nu_alpha}
\end{equation}
Let
\begin{equation}
\sum_{i=1}^\infty \delta_{(j_i,v_i,u_i)} \sim \PRM(\nu_\alpha \otimes \nu
\otimes \zeta) \label{PRM:underlying}
\end{equation}
be a Poisson random measure on $([-\infty,\infty] \setminus \{0\}) \times W
\times \mathbb{Z}^d$ with mean measure $\nu_\alpha \otimes \nu
\otimes \zeta$. Then from \eqref{repn_mixed_moving_avg}, it follows that
$\mathbf{X}$ has the following series representation:
$
X_t \eqdef {C_\alpha}^{1/\alpha} \sum_{i=1}^\infty j_i f(v_i,u_i-t),\;\;t \in
\mathbb{Z}^d
$,
where $C_\alpha$ is the stable tail constant given by
\begin{equation}
C_\alpha = {\left(\int_0^\infty x^{-\alpha} \sin{x}\,dx
\right)}^{-1}
 =\left\{
  \begin{array}{ll}
  \frac{1-\alpha}{\Gamma(2-\alpha) \cos{(\pi
\alpha/2)}},&\mbox{\textit{\small{if }}}\alpha \neq 1,\\
  \frac{2}{\pi},
&\mbox{\textit{\small{if }}}\alpha = 1.
  \end{array}
 \right. \label{defn:C_alpha}
\end{equation}
For simplicity of notations, we shall drop the factor ${C_\alpha}^{1/\alpha}$ and redefine $X_t$ as
\begin{equation}
X_t :=\sum_{i=1}^\infty j_i f(v_i,u_i-t),\;\;t \in
\mathbb{Z}^d\,. \label{repn_Possion_integral_X_n}
\end{equation}

 Mimicking the arguments given in \cite{resnick:samorodnitsky:2004}, it was established in Theorem~3.1 of \cite{roy:2010a} that the weak convergence
\begin{eqnarray*} \label{Roy:point process}
\sum_{\|t\|_\infty \leq n} \delta_{(2n)^{-d/\alpha} X_t} \Rightarrow \sum_{i=1}^\infty \sum_{u \in \mathbb{Z}^d} \delta_{j_i f(v_i,u)} \quad\mbox{ as } n\to\infty
\end{eqnarray*}
holds on the space of Radon point measures on $[-\infty,\infty] \setminus \{0\}$ equipped with the vague topology. Clearly the above limit is a cluster Poisson process.

For each $q \in\bbn_0$,  define a random vector field 
\begin{equation}
\widetilde{X}^q_t:=\{X_{t-w}\}_{w \in [-q\mathbf{1}_d, q\mathbf{1}_d]}. \label{defn:of:tilde:X}
\end{equation}
We take a sequence $\gamma_n$ satisfying $n^{d/\alpha}/\gamma_n \to 0$ so that for all $q\geq 0$,
\begin{equation}
N^{q}_n:=\sum_{\|t\|_\infty \leq n} \delta_{(n^{-1}t,\,\gamma_n^{-1}\widetilde{X}^q_t)}   \label{defn:of:N_n}
\end{equation}
converges almost surely to $\O$, the null measure in the space $\mathbb{M}^q$ defined in Section~\ref{subsec:HLS_conv}. We define a map
$
\psi: ([-\infty, \infty] \setminus \{0\}) \times W \times \mathbb{Z}^d \to [-\infty, \infty]^{[-q\mathbf{1}_d, q\mathbf{1}_d]}
$
by
\begin{equation}
\psi(x, v, u)= \{xf(v,u-w)\}_{w\in [-q\mathbf{1}_d, q\mathbf{1}_d]}     \label{defn:of:psi}
\end{equation}
in order to state the following result, which is an extension of Theorem~4.1 in \cite{hult:samorodnitsky:2010} to mixed moving average stable random fields. In particular, it describes the large deviation behavior of point processes induced by such fields.

\begin{Theorem} \label{thm:main:proc:diss}
Let $\{X_t\}_{t\in\mathbb{Z}^d}$ be the stationary symmetric $\alpha$-stable  mixed moving average  random field defined by \eqref{repn_Possion_integral_X_n} and $N^q_n$ be as in \eqref{defn:of:N_n} with
\begin{eqnarray} \label{gamma}
    n^{d/\alpha}/\gamma_n \to 0 \quad \mbox{ as } n\to\infty.
\end{eqnarray}
  Then for all $q \geq 0$, the HLS convergence
\begin{equation}
m^q_n(\cdot):=\frac{\gamma_n^\alpha}{n^d}\bbP(N^q_n \in \cdot) \rightarrow m^q_\ast(\cdot) \quad \mbox{ as } n\to\infty, \label{conv:m_n}
\end{equation}
holds in the space $\mathbf{M}_0(\mathbb{M}_p^q)$, where $m^q_\ast$ is a measure on $\mathbb{M}_p^q$ defined by
\begin{align*}
m^q_\ast(\cdot):=&(\Leb|_{[-1,1]^d} \otimes \nu_\alpha \otimes \nu)\Big( \Big\{(t,x,v) \in [-1,1]^d \times ([-\infty,\infty] \setminus \{0\}) \times W: \\
&\;\;\;\;\;\;\;\;\;\;\;\;\;\;\;\;\;\;\;\;\;\;\;\;\;\;\;\;\;\;\;\;\;\;\;\;\;\;\;\;\;\;\;\;\;\;\;\;\;\;\;\;\;\;\;\;\;\;\;\;\;\;\;\;\;\;\;\;\;\;\;\;\;\;\sum_{u \in \mathbb{Z}^d} \delta_{\left(t,\,\psi(x, v, u) \right)} \in \cdot\Big\}\Big)
\end{align*}
and satisfying $m^q_\ast(\mathbb{M}_p^q \setminus B(\O,\varepsilon))<\infty$ for all $\varepsilon>0$.
\end{Theorem}


The proof of the above result is given in the next section. The following statement is a direct  consequence of \cref{thm:main:proc:diss} in a similar pattern as in \cite{hult:samorodnitsky:2010}.

\begin{cor}\label{Corollary:Order Statistics}
Let $X_{i:n}$ be the $i$-th order statistic of \linebreak $\{X_t\}_{t\in [-n\mathbf{1}_d, n\mathbf{1}_d]}$ in
descending order, i.e., $X_{1:n}\geq X_{2:n}\geq \ldots \geq X_{(2n+1)^d\,:\,n}$\,.
Moreover, for all $v \in W$, let  $f_i^+(v)$ be the $i$-th
    order statistic of the sequence $\{f^+(v,u)\}_{u\in\bbz^d}$ in descending order
    and $f_i^-(v)$ be the $i$-th
    order statistic of the sequence $\{f^-(v,u)\}_{u\in\bbz^d}$ in descending order. Then
    for $y_1,\ldots,y_m> 0$,
    \begin{eqnarray*}
       \lefteqn{\lim_{n\to\infty}\frac{\gamma_n^{\alpha}}{n^d}
            \bbP(X_{1:n}>\gamma_ny_1, X_{2:n}>\gamma_ny_2,\ldots,X_{m:n}>\gamma_ny_m)}\\
        &&=2^d\int_W\Big(\min_{i=1,\ldots,m}(f_i^+(v)y_i^{-1})^{\alpha}+\min_{i=1,\ldots,m}(f_i^-(v)y_i^{-1})^{\alpha}\Big)
                    \nu({\rm d}v).
    \end{eqnarray*}
    In particular, for all $a>0$ and $n \geq 1$, if we define $\tau^a_n:=\inf\{\|t\|_{\infty}: X_t>a\gamma_n\}$, then
    \begin{eqnarray*}
        \lim_{n\to\infty}\frac{\gamma_n^{\alpha}}{n^d}
            \bbP(\tau^a_n\leq \lambda n)
        = (2\lambda)^d a^{-\alpha}\int_W(\sup_{u\in\bbz^d}f^+(v,u))^{\alpha}+(\sup_{u\in\bbz^d}f^-(v,u))^{\alpha}
                    \nu({\rm d}v).
    \end{eqnarray*}
\end{cor}

\begin{proof}[\textbf{Proof}]
Following the proof of Corollary~5.1 in \cite{hult:samorodnitsky:2010}, we can show that the set
$$B(y_1,y_2, \ldots,y_m):=\bigcap_{i=1}^m\left\{\xi\in \mathbb{M}_p^0:\xi([-1,1]^d\times (y_i,\infty))\geq i\right\}$$ is bounded away from the null measure and its boundary is an $m^0_\ast$-null set. Therefore by applying \cref{thm:main:proc:diss} with $q=0$ and Portmanteau-Theorem (Theorem 2.4 in \cite{hult:lindskog:2006}), we obtain
\begin{eqnarray*}
        \lim_{n\to\infty}\frac{\gamma_n^{\alpha}}{n^d}
            \bbP(X_{1:n}>\gamma_ny_1,\ldots,X_{m:n}>\gamma_ny_m)
        &=&\lim_{n\to\infty}m_n^0(B(y_1,\ldots,y_m))\\
        &=&m_\ast^0(B(y_1,\ldots,y_m)),
\end{eqnarray*}
which can be shown to be equal to the first limit above by an easy calculation.


The second statement follows trivially from the first one using the observation that
\begin{eqnarray*}
    \frac{\gamma_n^{\alpha}}{n^d}
            \bbP(\tau^a_n\leq \lambda n)&=&
    \frac{\gamma_n^{\alpha}}{n^d}
            \bbP\left(\sup_{t\in[-\lfloor n\lambda \rfloor\mathbf{1}_d,\lfloor n\lambda \rfloor\mathbf{1}_d]} X_t>a\gamma_n\right)
\end{eqnarray*}
for all $n \geq 1$ and $a>0$.
\end{proof}

\subsection{Proof of Theorem~\ref{thm:main:proc:diss}}
We shall first discuss a brief sketch of the proof of Theorem~\ref{thm:main:proc:diss}. Fix Lipschitz functions $g_1,\,g_2 \in C^+_K(\mathbb{E}^q)$ and  $\epsilon_1,\,\epsilon_2>0$.  By Theorem~A.2 of \cite{hult:samorodnitsky:2010}, in order to prove \eqref{conv:m_n}, it is
enough to show that $m^q_\ast \in \mathbf{M}_0(\mathbb{M}_p^q)$ and
\begin{equation}
\lim_{n \to \infty} m^q_n(F_{g_1,g_2,\epsilon_1,\epsilon_2}) = m^q_\ast(F_{g_1,g_2,\epsilon_1,\epsilon_2}) \label{suff:condn:conv:mq_n}
\end{equation}
with $F_{g_1,g_2,\epsilon_1,\epsilon_2}$ as in \eqref{eq:F}.
Following the heuristics in \cite{resnick:samorodnitsky:2004}, one expects that under the normalization $\gamma_n^{-1}$, all the Poisson points in \eqref{repn_Possion_integral_X_n} except perhaps one will be killed and therefore the large deviation behavior of $N^q_n$ should be the same as that of
$$
\widehat{N}^q_n:=\sum_{i=1}^\infty \sum_{\|t\|_\infty \leq n} \delta_{(n^{-1}t,\,\gamma_n^{-1}\psi(j_i, v_i, u_i-t))}.
$$ Keeping this in mind, we define
\begin{eqnarray} \label{eq:mhat}
    \widehat{m}^q_n(\cdot):=\frac{\gamma_n^\alpha}{n^d}\bbP(\widehat{N}^q_n \in \cdot)
\end{eqnarray}
and hope to establish
\begin{equation}
\lim_{n \to \infty} \widehat{m}^q_n(F_{g_1,g_2,\epsilon_1,\epsilon_2}) = m^q_\ast(F_{g_1,g_2,\epsilon_1,\epsilon_2}).  \label{conv:hatm_n}\,
\end{equation}
as the first step of proving \eqref{suff:condn:conv:mq_n}.

For $p=1,2$ and for all $i\in\mathbb{N}$, let
\begin{align}
Z_{p,i}:=\sum_{\|t\|_\infty \leq n} g_p(n^{-1}t, \gamma_n^{-1}\psi(j_i, v_i, u_i-t)), \label{defn:Z_p,i}
\end{align}
where $\psi$ is as in \eqref{defn:of:psi}. For all $q \geq 0$ and $n \geq 1$, define
\begin{eqnarray} \label{eq:mtilde32}
\widetilde{m}^q_n(F_{g_1,g_2,\epsilon_1,\epsilon_2}):=\frac{\gamma_n^\alpha}{n^d}\bbE\Big[\sum_{i=1}^\infty \big(1-\e^{-(Z_{1,i}-\epsilon_1)_+}\big)\big(1-\e^{-(Z_{2,i}-\epsilon_1)_+}\big)\Big].
\end{eqnarray}
In order to establish \eqref{conv:hatm_n}, we shall first show that the quantities $\widehat{m}^q_n(F_{g_1,g_2,\epsilon_1,\epsilon_2})$ and $\widetilde{m}^q_n(F_{g_1,g_2,\epsilon_1,\epsilon_2})$ are asymptotically equal, and then prove
\begin{equation}
\lim_{n \to \infty} \widetilde{m}^q_n(F_{g_1,g_2,\epsilon_1,\epsilon_2}) = m^q_\ast(F_{g_1,g_2,\epsilon_1,\epsilon_2}).  \nonumber
\end{equation}

The execution and justification of these steps are detailed below with the help of a series of lemmas. Among these, Lemma~\ref{lemma:diff:of:mhat:and:mtilde} is the key step that makes our proof amenable to the techniques used in \cite{resnick:samorodnitsky:2004}. The rest of the lemmas can be established by closely following the proof of Theorem~3.1 in the aforementioned paper and improving it whenever necessary. Most of these improvements are nontrivial albeit somewhat expected.

The first step in establishing the HLS convergence \eqref{conv:m_n} is to check that the limit measure $m^q_\ast$ is indeed an element $\mathbf{M}_0(\mathbb{M}_p^q)$.

\begin{lemma} \label{lemma:diss:1}
    For all $q \geq 0$, $m^q_\ast \in \mathbf{M}_0(\mathbb{M}_p^q).$
\end{lemma}
\begin{proof}[\textbf{Proof}]
The statement $m^q_\ast \in \mathbf{M}_0(\mathbb{M}_p^q)$ means that
$m^q_\ast$ is a Borel measure on $\mathbb{M}_p^q$ with $m^q_\ast(\mathbb{M}_p^q\backslash B({\O},\epsilon))<\infty$ for
any $\epsilon>0$.
To prove this, we first claim that for almost all $(t, x, v) \in  [-1,1]^d \times ([-\infty,\infty]\setminus \{0\}) \times W$,
\begin{equation}
\sum_{u \in \mathbb{Z}^d}\delta_{\left(t,\,\psi(x,v,u)\right)}  \in \mathbb{M}_p^q, \label{first:step:comp:of:m_ast}
\end{equation}
concluding $m^q_\ast$ is a Borel measure on $\mathbb{M}_p^q$. To this end, setting
\begin{equation} 
A_\eta:=[-\infty,\infty]^{[-q\mathbf{1}_d, q\mathbf{1}_d]}\setminus(-\eta,\eta)^{[-q\mathbf{1}_d, q\mathbf{1}_d]} \label{defn:A_eta}
\end{equation}
for all $\eta>0$, and $\|f\|_\alpha:=\left(\int_W \sum_{u\in\mathbb{Z}^d}|f(v,u)|^\alpha\nu(dv)\right)^{1/\alpha}$, we get
\begin{align}
&\int_{[-1,1]^d} \int_{|x|>0} \int_W \sum_{u \in \mathbb{Z}^d}\delta_{\left(t,\,\psi(x,v,u)\right)}\left([-1,1]^d \times A_\eta\right) \nu(dv) \nu_\alpha(dx) dt \nonumber\\
&\quad \leq 2^{d+1}\eta^{-\alpha} (2q+1)^d\, \|f\|_\alpha^\alpha <\infty. \label{eq.3.20}
\end{align}
Applying the method used to establish that the limit measure in Theorem~3.1 of \cite{resnick:samorodnitsky:2004}  (p.196) is Radon,
\eqref{first:step:comp:of:m_ast} follows from \eqref{eq.3.20}.

Because of the estimates used in the proof of Theorem~A.2 in \linebreak \cite{hult:samorodnitsky:2010}, to obtain  $m^q_\ast(\mathbb{M}_p^q\backslash B({\O},\epsilon))<\infty$ for all $\epsilon > 0$, it is enough to show that
$
m^q_\ast(F_{g_1,g_2,\epsilon_1,\epsilon_2}) < \infty
$
for all $g_1,\,g_2 \in C^+_K(\mathbb{E}^q)$ and for all $\epsilon_1,\,\epsilon_2>0$.
Using \eqref{first:step:comp:of:m_ast} and a change of measure, we get
\begin{align}
m^q_\ast(F_{g_1,g_2,\epsilon_1, \epsilon_2})
&=\int_{[-1,1]^d} \int_{|x|>0} \int_W \Big\{\left(1-\e^{-(\sum_{u\in\mathbb{Z}^d}\,g_1(t,\psi(x,v,u))-\epsilon_1)_{+}}\right) \nonumber\\
&\hspace{0.25in}\times \left(1-\e^{-(\sum_{u\in\mathbb{Z}^d}\,g_2(t,\psi(x,v,u))-\epsilon_2)_{+}}\right)\Big\}\nu(dv) \nu_\alpha(dx) dt.\nonumber
\end{align}
Let $C$ be an upper bound for $|g_1|$ and $|g_2|$, and $\eta > 0$ be  such that $g_1(t,y)=g_2(t,y)=0$ for all $y \in (-\eta,\eta)^{[-q\mathbf{1}_d, q\mathbf{1}_d]}$. Then \eqref{eq.3.20} and the inequality $1-\e^{-(x-\epsilon)_{+}} \leq x$ (for $x \geq 0$ and $\epsilon >0$) yield that $m^q_\ast(F_{g_1,g_2,\epsilon_1,\epsilon_2})$ can be bounded by $2^{d+1}C \eta^{-\alpha}(2q+1)^d\|f\|_\alpha^\alpha $. This shows $m^q_\ast(\mathbb{M}_p^q\backslash B({\O},\epsilon))<\infty$.
\end{proof}

To proceed with the proof of \cref{thm:main:proc:diss} by using the ideas mentioned above, we need the following most crucial lemma.

\begin{lemma} \label{lemma:diff:of:mhat:and:mtilde}
Let $\widehat{m}^q_n(F_{g_1,g_2,\epsilon_1,\epsilon_2})$ and $\widetilde{m}^q_n(F_{g_1,g_2,\epsilon_1,\epsilon_2})$ be as in \eqref{eq:mhat} and \eqref{eq:mtilde32}, respectively. Then for all $q \geq 0$,
\[
\lim_{n\to\infty}|\widehat{m}^q_n(F_{g_1,g_2,\epsilon_1,\epsilon_2})-\widetilde{m}^q_n(F_{g_1,g_2,\epsilon_1,\epsilon_2})|=0.
\]
\end{lemma}

\begin{proof}[\textbf{Proof}] Let $C, \eta >0$ be as above and $A_\eta$ be defined by \eqref{defn:A_eta}. For $n \geq 1$, let $B_n$ be the event that for at most one $i$,
$
\sum_{\|t\|_\infty \leq n} \delta_{\gamma_n^{-1}\psi(j_i, v_i, u_i-t)}(A_\eta) \geq 1,
$
where $\psi$ is as in \eqref{defn:of:psi}. We claim that
\begin{equation}
\frac{\gamma_n^\alpha}{n^d} \bbP(B_n^c) \to 0 \label{bound:on:Prob:of:B_n:compliment}
\end{equation}
as $n \to \infty$. To prove this claim, observe that on $B_n^c$, there exist more than one $i$ such that $|j_i| \geq \eta\gamma_n/|f(v_i,u_i-t-w)|$ for some $(t,w) \in [-n\mathbf{1}_d, n\mathbf{1}_d] \times [-q\mathbf{1}_d, q\mathbf{1}_d]$ and therefore because of \eqref{PRM:underlying}, the sequence in \eqref{bound:on:Prob:of:B_n:compliment} can be bounded by
\[
\frac{\gamma_n^\alpha}{n^d} \bbP\bigg(\sum_{i=1}^\infty \delta_{(j_i,v_i,u_i)}(L_n) \geq 2\bigg) \leq \frac{\gamma_n^\alpha}{n^d} \bigg(\bbE\Big(\sum_{i=1}^\infty \delta_{(j_i,v_i,u_i)}(L_n)\Big)\bigg)^2 = O\left({n^{d}}/{\gamma_n^{\alpha}}\right),
\]
where
$
L_n:=\left\{(x,v,u): |x| \geq \eta \gamma_n\big(\sum_{\|t\|_\infty \leq n} \sum_{\|w\|_\infty \leq q}|f(v,u-t-w)|^\alpha\big)^{-\frac{1}{\alpha}}\right\}
 $. 

It is easy to check that with $Z_{1,i}$ and $Z_{2,i}$ as in \eqref{defn:Z_p,i},  \linebreak
$\widehat{m}^q_n(F_{g_1,g_2,\epsilon_1,\epsilon_2}) =\frac{\gamma_n^\alpha}{n^d} \bbE \Big[\big(1-\e^{-(\sum_{i=1}^\infty Z_{1,i}-\epsilon_1)_+}\big) \big(1-\e^{-(\sum_{i=1}^\infty Z_{2,i}-\epsilon_2)_+}\big)\Big]$. Since on the event $B_n$, the random variables $\big(1-\e^{-(\sum_{i=1}^\infty Z_{1,i}-\epsilon_1)_+}\big) \big(1-\e^{-(\sum_{i=1}^\infty Z_{2,i}-\epsilon_2)_+}\big)$ and $\sum_{i=1}^\infty \big(1-\e^{-(Z_{1,i}-\epsilon_1)_+}\big)\big(1-\e^{-(Z_{2,i}-\epsilon_1)_+}\big)$ are equal, it transpires that
\begin{eqnarray*}
\lefteqn{|\widehat{m}^q_n(F_{g_1,g_2,\epsilon_1,\epsilon_2})-\widetilde{m}^q_n(F_{g_1,g_2,\epsilon_1,\epsilon_2})|} \nonumber\\
& \leq& \frac{\gamma_n^\alpha}{n^d} \bbP(B_n^c) + \frac{\gamma_n^\alpha}{n^d} \bbE \Big[\1_{B_n^c}\sum_{i=1}^\infty \big(1-\e^{-(Z_{1,i}-\epsilon_1)_+}\big) \big(1-\e^{-(Z_{2,i}-\epsilon_2)_+}\big)\Big]\nonumber\\
&\leq& \frac{\gamma_n^\alpha}{n^d} \bbP(B_n^c) + \sqrt{\frac{\gamma_n^\alpha}{n^d}\bbP(B_n^c)\,\frac{\gamma_n^\alpha}{n^d}\bbE\,\Big(\sum_{i=1}^\infty \big(1-\e^{-(Z_{1,i}-\epsilon_1)_+}\big) \Big)^2}\;,\nonumber
\end{eqnarray*}
which, combined with \eqref{bound:on:Prob:of:B_n:compliment}, yields Lemma~\ref{lemma:diff:of:mhat:and:mtilde} provided we show that
\begin{equation}
\frac{\gamma_n^\alpha}{n^d}\bbE\,\Big(\sum_{i=1}^\infty \big(1-\e^{-(Z_{1,i}-\epsilon_1)_+}\big) \Big)^2=O(1). \label{bigOone}
\end{equation}

To this end, note that applying \eqref{PRM:underlying}, Lemma 9.5IV in \cite{Daley:Vere-JonesII},
 and the inequality $1 - \e^{-x} \leq x$ for $x \geq 0$, we obtain
\begin{align*}
&\bbE\,\Big(\sum_{i=1}^\infty \big(1-\e^{-(Z_{1,i}-\epsilon_1)_+}\big) \Big)^2\nonumber\\
&=  \int_{|x|>0} \int_W\sum_{u \in \mathbb{Z}^d}\big(1-\e^{-(\sum_{\|t\|_\infty \leq n} \,g_1(n^{-1}t, \gamma_n^{-1}\psi(x, v, u-t))-\epsilon_1)_+}\big)^2 \nu(dv) \nu_\alpha(dx)\nonumber\\
& \;\;\;\;+ \bigg(\int_{|x|>0} \int_W\sum_{u \in \mathbb{Z}^d}\big(1-\e^{-(\sum_{\|t\|_\infty \leq n} g_1(n^{-1}t, \gamma_n^{-1}\psi(x, v, u-t))-\epsilon_1)_+} \big)\\
&\hspace{3.7in}\nu(dv) \nu_\alpha(dx)\bigg)^2 \nonumber\\
&\leq \int_{|x|>0} \int_W \sum_{u \in \mathbb{Z}^d}  \sum_{\|t\|_\infty \leq n}\,g_1\big(n^{-1}t, \gamma_n^{-1}\psi(x, v, u-t)\big) \nu(dv) \nu_\alpha(dx) \nonumber\\
&\;\;\;\;+ \bigg(\int_{|x|>0} \int_W \sum_{u \in \mathbb{Z}^d} \sum_{\|t\|_\infty \leq n}\,g_1\big(n^{-1}t, \gamma_n^{-1}\psi(x, v, u-t)\big) \nu(dv) \nu_\alpha(dx)\bigg)^2,\nonumber
\end{align*}
from which \eqref{bigOone} follows because by similar calculations as in \eqref{eq.3.20}, the first term of above is bounded by $2 C (\eta \gamma_n)^{-\alpha}(2q+1)^d \|f\|_\alpha^\alpha (2n+1)^d$ for all $n \geq 1$ and $q \geq 0$ and for the second term we additionally use \eqref{gamma}. This finishes the proof of this lemma.
\end{proof}

We shall now establish \eqref{conv:hatm_n}. In light of Lemma~\ref{lemma:diff:of:mhat:and:mtilde}, it is enough to prove the next lemma.
\begin{lemma}
 For all $q \geq 0$,
\begin{equation}
\lim_{n\to\infty}\widetilde{m}^q_n(F_{g_1,g_2,\epsilon_1,\epsilon_2}) = m^q_\ast(F_{g_1,g_2,\epsilon_1,\epsilon_2}).  \label{conv:tildem_n}
\end{equation}
\end{lemma}
\begin{proof}[\textbf{Proof}]
This can be achieved in a fashion similar to the proof of Theorem~3.1 in \linebreak \cite{resnick:samorodnitsky:2004}, namely, by first proving a version of \eqref{conv:tildem_n} for $f$ supported on $W \times [-T\mathbf{1}_d,T\mathbf{1}_d]$ for some $T \geq 1$, and then using a converging together argument with the help of the inequalities used in the proof of Lemma~\ref{diff:of:m_F:and:mhat_F} below.
\end{proof}
Therefore in order to complete the proof of Theorem~\ref{thm:main:proc:diss}, it remains to establish the following lemma.

\begin{lemma} \label{diff:of:m_F:and:mhat_F} For all $q \geq 0$,
\begin{equation*}
\lim_{n \to \infty} \big|m^q_n(F_{g_1,g_2,\epsilon_1,\epsilon_2})-\widehat{m}^q_n(F_{g_1,g_2,\epsilon_1,\epsilon_2}) \big| = 0.
\end{equation*}
\end{lemma}
\begin{proof}[\textbf{Proof}]
Because of the inequalities $|x_1 x_2 - y_1 y_2| \leq |x_1-y_1|+|x_2-y_2|$ for $x_1, x_2, y_1, y_2 \in [0,1]$ and $|\e^{-(z_1-\epsilon_1)_+}-\e^{-(z_2-\epsilon_2)_+}| \leq |z_1 - z_2|$ for $z_1, z_2 \in \left[0,\infty\right)$ and $\epsilon_1, \epsilon_2 \in (0, \infty)$, the convergence in \Cref{diff:of:m_F:and:mhat_F} will be established provided we show that for all Lipschitz $g \in C^+_K(\mathbb{E}^q)$,
\begin{equation}
\frac{\gamma_n^\alpha}{n^d} \bbE\big|N_n^q(g) - \widehat{N}_n^q(g)\big| \to 0 \label{diff:of:N_g:and:Nhat_g}
\end{equation}
as $n \to \infty$. We shall establish \eqref{diff:of:N_g:and:Nhat_g} by closely following the proof of (3.14) in \cite{resnick:samorodnitsky:2004} and modifying their estimates as needed. We sketch the main steps below.

Assume that $|g| \leq C$ and $g(t,y)=0$ for all $y \in (-\eta,\eta)^{[-q\mathbf{1}_d, q\mathbf{1}_d]}$. For each $n\geq 1$ and for each $\theta >0$, let $A(\theta, n)$ denote the event that for all $\|t\|_\infty \leq n$ and for all $\|w\|_\infty \leq q$, $\sum_{i=1}^\infty \delta_{|j_if(v_i,u_i-t-w)|}\big([\gamma_n \theta, \infty]\big)\leq 1$.
Then similarly as in \cite{resnick:samorodnitsky:2004} p.201 it follows that for all $\theta >0$, 
\begin{equation}
    \gamma_n^\alpha \bbP \big(A(\theta,n)^c\big) \to 0  \quad \mbox{ as }n \to \infty.\label{bound:on:Prob:A_theta,n^c}
\end{equation}
 Defining $Y_t$ to be the summand of largest modulus in
$
X_t = \sum_{i=1}^\infty j_i f(v_i, u_i-t)
$
 for all $t \in \mathbb{Z}^d$, and adapting the method of \cite{resnick:samorodnitsky:2004} p.~201 to our situation, we can find $T \in \mathbb{N}$ such that for all $\theta < \eta/2$,
\begin{equation}
D(\theta,n):=\bigg\{\bigvee_{\|w\|_\infty \leq q}\,\bigvee_{\|t\|_\infty \leq n} \left|\gamma_n^{-1}X_{t-w} - \gamma_n^{-1}Y_{t-w} \right|>\theta \bigg\}\cap A\left(\theta/T, n\right) \nonumber
\end{equation}
satisfies
\begin{equation} \label{eq.3.26}
    \lim_{n\to\infty}\gamma_n^\alpha \bbP\big(D(\theta,n)\big)= 0.
\end{equation}

Define, for each $q \geq 0$, a random vector field $\{\widetilde{Y}^q_t\}_{t \in \mathbb{Z}^d}$ in $ \mathbb{R}^{[-q\mathbf{1}_d, q\mathbf{1}_d]}$ by replacing $\{X_t\}_{t \in \mathbb{Z}^d}$ by $\{Y_t\}_{t \in \mathbb{Z}^d}$ in \eqref{defn:of:tilde:X}. For any $\theta < \eta/2$, the sequence in \eqref{diff:of:N_g:and:Nhat_g} is bounded by
\begin{align}
& \frac{\gamma_n^\alpha}{n^d} \sum_{\|t\|_\infty \leq n} \bbE\,\big|g(n^{-1}t, \gamma_n^{-1} \widetilde{X}^q_t)-g(n^{-1}t, \gamma_n^{-1} \widetilde{Y}^q_t)\big| \1_{D(\theta,n)}\nonumber\\
& \;\;\;\;+ \frac{\gamma_n^\alpha}{n^d} \sum_{\|t\|_\infty \leq n} \bbE\,\big|g(n^{-1}t, \gamma_n^{-1} \widetilde{X}^q_t)-g(n^{-1}t, \gamma_n^{-1} \widetilde{Y}^q_t)\big|  \1_{A(\theta/M,n) \setminus D(\theta,n)}\nonumber\\
& \;\;\;\;+ \frac{\gamma_n^\alpha}{n^d} \bbE\, \big|N_n^q(g)\big|\1_{A(\theta/M,n)^c} + \frac{\gamma_n^\alpha}{n^d} \bbE\, \big|\widehat{N}_n^q(g)\big|\1_{A(\theta/M,n)^c} \nonumber\\
&=\frac{\gamma_n^\alpha}{n^d} \sum_{\|t\|_\infty \leq n} \bbE\,\big|g(n^{-1}t, \gamma_n^{-1} \widetilde{X}^q_t)-g(n^{-1}t, \gamma_n^{-1} \widetilde{Y}^q_t)\big|  \1_{A(\theta/M,n) \setminus D(\theta,n)} \nonumber\\
&\;\;\;\;+\frac{\gamma_n^\alpha}{n^d} \bbE\, \big|\widehat{N}_n^q(g)\big|\1_{A(\theta/M,n)^c}  \;+\;o(1). \nonumber
\end{align}
In the last step, we used the asymptotic results \eqref{bound:on:Prob:A_theta,n^c} and \eqref{eq.3.26}, and the fact that $g$ bounded.
Following \cite{resnick:samorodnitsky:2004} p.~202 , the first term above can be bounded by $2 L_g (\eta/2)^{-\alpha} (2q+1)^d \|f\|_\alpha^\alpha \left(\frac{2n+1}{n}\right)^d \theta$ (here $L_g$ denotes the Lipschitz constant of $g$) and repeating the method used in the proof of Lemma~\ref{lemma:diff:of:mhat:and:mtilde}, the second term can be shown to be $o(1)$. Since $\theta \in (0, \eta/2)$ is arbitrary, \eqref{diff:of:N_g:and:Nhat_g} follows.
\end{proof}



\section{The conservative case} \label{section:conservative}

Suppose now that $\mathbf{X}$ is a stationary $S\alpha S$ random field generated by a conservative $\mathbb{Z}^d$-action. Unlike the mixed moving average representation in the dissipative case, no nice representation is available in general. However, if we view the underlying action as a group of invertible nonsingular transformations on $(S,\mathcal{S},\mu)$ (see \cite{roy:samorodnitsky:2008} and \cite{roy:2010a}), then under certain conditions,  $\mathbf{X}$ can be thought of as a lower dimensional mixed moving average field. This will enable us to analyze the large deviation issues of point processes induced by such fields.

Let $A:=\{\phi_t:\,t \in \mathbb{Z}^d\}$ be the subgroup of the
group of invertible nonsingular transformations on $(S,\mathcal{S},\mu)$ and
$
\Phi:\mathbb{Z}^d \rightarrow A
$
be a group homomorphism defined by $\Phi(t)=\phi_t$ for all $t \in \mathbb{Z}^d$ with kernel
$K:=Ker(\Phi)=\{t \in \mathbb{Z}^d:\,\phi_t = 1_S\}$. Here $1_S$
is the identity map on $S$. By the first isomorphism theorem of groups (see, for example,
\cite{lang:2002}) we have $A \simeq \mathbb{Z}^d/K$. Therefore, the structure theorem for finitely generated abelian groups (see Theorem $8.5$ in Chapter I of \cite{lang:2002}) yields
$
A=\overline{F} \oplus \overline{N}\,,
$
where $\overline{F}$ is a free abelian group and $\overline{N}$ is a finite
group. Assume $rank(\overline{F})=p \geq 1$ and $|\overline{N}|=l$. Since
$\overline{F}$ is free, there exists an injective group homomorphism
$
\Psi: \overline{F} \rightarrow \mathbb{Z}^d
$
such that $\Phi \circ \Psi$ is the identity map on $\overline{F}$.

Clearly, $F:=\Psi(\overline{F})$ is a free subgroup of $\mathbb{Z}^d$ of rank $p \leq d$. It follows easily that the sum $F+K$ is direct and
$
\mathbb{Z}^d/(F+K) \simeq  \overline{N}
$.
Let $x_1+(F +K)$,
$x_2+(F +K),\,\ldots\,,x_l+(F+K)$ be all the cosets of $F +K$ in $\mathbb{Z}^d$. It has been observed in \cite{roy:samorodnitsky:2008} that
$
H:=\bigcup_{k=1}^l (x_k + F) \label{defn_of_H}
$
forms a countable Abelian group (isomorphic to $\mathbb{Z}^d/K$) under addition $\oplus$ modulo $K$ [for all $s_1,s_2\in H$, $s_1\oplus s_2$ is defined as the unique $s\in H$ such that $(s_1+s_2)-s\in K$]
and it admits a map $N:H \to
\{0,1,\ldots\}$ defined by \linebreak
$
N(s):=\min\{\|s+v\|_ \infty: v \in K\}
$
satisfying symmetry [for all $s \in H$, \linebreak $N(s^{-1})=N(s)$, where $s^{-1}$ is the inverse of $s$ in $(H,\oplus)$] and triangle inequality [for all $s_1, s_2
\in H$, $N(s_1 \oplus s_2) \leq N(s_1) + N(s_2)$]. Note that every $t \in \Zd$ can be decomposed uniquely as $t=t_H+t_K$, where $t_H \in H$ and $t_K \in K$.
Therefore, we can define a projection map $\pi: \Zd \to H$ as $\pi(t)=t_H$ for all $s \in \Zd$.

Define, for all $n \geq 1$,
$
H_n=\{s \in H: N(s) \leq n\}.
$
It is easy to see that $H_n$'s are finite subsets increasing to $H$ and
\begin{equation}
|H_n| \sim c n^p \label{rate_of_growth:H_n} \quad \mbox{ as }n\to\infty,
\end{equation}
for some $c>0$; see (5.19) in \cite{roy:samorodnitsky:2008}. If $\{\phi_t\}_{t \in F}$ is a dissipative group action then $\{\wt\phi_s\}_{s \in H}$ defined by
$
\wt\phi_s= \phi_s
$
is a dissipative $H$-action; see, once again, \cite{roy:samorodnitsky:2008}, p.228. Because of Remark 4.3 in \cite{roy:2010a} (an extremely useful observation of Jan Rosi\'nski), without loss of generality, all the known examples of stationary $S \alpha S$ random fields can be assumed to satisfy
\begin{equation}
c_v \equiv 1 \;\;\;\mbox{ for all }v \in K,
\label{assumption_on_c_t_for_t_in_K}
\end{equation}
which would immediately yield that $\{c_s\}_{s \in H}$ is an $H$-cocycle for
$\{\wt\phi_s\}_{s \in H}$. Hence the subfield $\{X_s\}_{s \in H}$ is $H$-stationary and is
generated by the dissipative action $\{\wt\phi_s\}_{s \in H}$. This implies, in particular, that there is a standard Borel space
$(W,\mathcal{W})$ with a $\sigma$-finite measure $\nu$ on it such
that
\begin{equation}
X_s \eqdef \int_{W \times H} h(v,u \oplus s)\,M^\prime(dv,du),\;\;\;
s \in H, \label{mixed_moving_avg_repn_of_X_r_r_in_H}
\end{equation}
for some $h \in L^{\alpha}(W\times H, \nu \otimes \zeta_H)$, where
$\zeta_H$ is the counting measure on $H$, and  $M^\prime$ is a $S\alpha
S$ random measure on $W \times H$ with control measure $\nu \otimes
\zeta_H$ (see, for example, Remark $2.4.2$ in \cite{roy:2008}).
Let
$$
\sum_{i=1}^{\infty} \delta_{(j_i,v_i,u_i)} \sim \PRM(\nu_\alpha \otimes
\nu \otimes \zeta_H)
$$
be a Poisson random measure on $([-\infty,\infty] \setminus \{0\}) \times W
\times H$, where $\nu_\alpha(\cdot)$ is the measure defined by \eqref{defn:nu_alpha}.
The following series representation holds in parallel to \eqref{repn_Possion_integral_X_n} after dropping a factor of $C_\alpha^{1/\alpha}$ ($C_\alpha$ is as in \eqref{defn:C_alpha}):
\begin{equation*}
X_s = \sum_{i=1}^\infty j_ih(v_i,u_i\oplus
s),\;\;\;s \in H.
\end{equation*}
Note that $rank(K)=d-p$; see the proof of Proposition 3.1 in \linebreak \cite{chakrabarty:roy:2013}.
Assume $p<d$. Let $U$ be a $d \times p$ matrix whose columns form a basis of $F$ and $V$ be a $d \times (d-p)$ matrix whose columns form a basis of $K$. Let
\begin{equation}
\Delta:=\{y \in \mathbb{R}^p: \mbox{ there exists }\lambda \in
\mathbb{R}^{d-p} \mbox{ such that } \|Uy+V\lambda\|_\infty \leq 1 \}, \nonumber
\end{equation}
which is a compact and convex set; see Lemma~5.1 in \cite{roy:2010a}. For all $y \in
\Delta$, define
$
Q_y:=\{\lambda \in \mathbb{R}^{d-p}:\, \|Uy+V\lambda\|_\infty \leq 1 \}
$
and let $\mathcal{V}(y)$ be the $q$-dimensional volume of $Q_y$. Lemma~5.1 in \cite{roy:2010a}
says that $\mathcal{V}:\Delta\to [0,\infty)$ is a continuous map.

We also define a map
$
\psi_H: ([-\infty, \infty] \setminus \{0\}) \times W \times H\to [-\infty, \infty]^{[-q\mathbf{1}_d, q\mathbf{1}_d]}
$
by
\begin{equation*}
\psi_H(x, v, u)= \{xh(v,u\ominus \pi(w))\}_{ w \in [-q\mathbf{1}_d, q\mathbf{1}_d]},
\end{equation*}
where $\pi$ is the projection on $H$ as above and $u \ominus s:=u \oplus s^{-1}$ with $s^{-1}$ being the inverse of $s$ in $(H,\oplus)$.

The rank $p$ can be regarded as the effective dimension of the random field and it gives more precise information on the rate of growth of the partial maxima than the actual dimension $d$. More precisely, according to Theorem 5.4 in \cite{roy:samorodnitsky:2008},
\begin{equation*}
n^{-p/\alpha} \max_{\|t\|_\infty \leq n }|X_t| \Rightarrow \left\{
                                     \begin{array}{ll}
                                     c^\prime_\bX \xi_\alpha & \mbox{ if $\{\phi_t\}_{t \in F}$ is a dissipative action,} \\
                                     0              & \mbox{ if $\{\phi_t\}_{t \in F}$ is a conservative action},
                                     \end{array}
                              \right. 
\end{equation*}
where $c^\prime_\bX$ is a positive constant depending on $\bX$ and $\xi_\alpha$ is as in \eqref{cdf_of_Z_alpha}.


However, even when $\{\phi_t\}_{t \in F}$ is dissipative and \eqref{assumption_on_c_t_for_t_in_K} holds, the point process sequence $\sum_{\|t\|_\infty \leq n} \delta_{n^{-p/\alpha}X_t}$ does not remain tight due to clustering of points owing to the longer memory of the field. It so happens that the cluster sizes are of order $n^{d-p}$ and therefore the scaled point process
$
n^{p-d}\sum_{\|t\|_\infty \leq n} \delta_{n^{-p/\alpha}X_t}
$
converges weakly to a random measure on $[-\infty, \infty] \setminus \{0\}$; see Theorem 4.1 in \cite{roy:2010a}.
To be precise
\[
   \hspace*{-0.2cm} n^{p-d}\sum_{\|t\|_\infty \leq n} \delta_{( l \text{Leb}(\Delta)n^p)^{-1/\alpha}X_t}\weak\sum_{u \in H}\sum_{i=1}^{\infty}
    \mathcal{V}(\xi_i)\delta_{j_ih(v_i,u)} \quad \mbox{ as } n\to\infty,
\]
where $\sum_{i=1}^{\infty}\delta_{(\xi_i,j_i,v_i)} \sim $ \PRM$(\mbox{Leb}|_{\Delta}\otimes\nu_\alpha \otimes \nu)$.
Therefore, we take a sequence $\norm_n$ such that $ n^{p/\alpha}/ \norm_n \to 0$
as $n\to\infty$ so that for all $q \geq 0$ and $\widetilde{X}^q_t$ as defined in \eqref{defn:of:tilde:X},
\begin{equation}
\Lamb^q_n:=n^{p-d}\sum_{\|t\|_\infty \leq n} \delta_{(n^{-1}t, \norm_n^{-1}\widetilde{X}^q_t)} \label{defn:of:lambda_n}
\end{equation}
converges almost surely to $\O$. With the notations introduced above, we have the following result.

\begin{Theorem}\label{thm:main:proc:cons}
Let $\{X_t\}_{t\in\mathbb{Z}^d}$ be a stationary symmetric $\alpha$-stable random field generated by a conservative action $\{\phi_t\}_{t \in \mathbb{Z}^d}$ and $\Lamb^q_n$ be as in \eqref{defn:of:lambda_n} with
\begin{eqnarray} \label{beta}
    n^{\frac{p}{\alpha}}/\norm_n\to 0 \quad \mbox{ as } n\to\infty.
\end{eqnarray}
Assume $1 \leq p < d$, $\{\phi_t\}_{t \in F}$ is dissipative and \eqref{assumption_on_c_t_for_t_in_K} holds. Then for all $q \geq 0$, the HLS convergence
\begin{equation}
\ka^q_n(\cdot):=\frac{\norm_n^\alpha}{n^p}\bbP(\Lamb^q_n \in \cdot) \rightarrow \ka^q_\ast(\cdot) \label{conv:kappa_n}
    \quad \mbox{ as } n\to\infty
\end{equation}
holds in the space $\mathbf{M}_0(\mathbb{M}^q)$, where $\ka^q_\ast$ is a measure on $\mathbb{M}^q$ defined by
\begin{align*}
\ka^q_\ast(\cdot):=&\,l \,(\text{Leb}|_\Delta \otimes \nu_\alpha \otimes \nu) \Big(\Big\{(y,x,v) \in \Delta \times ([-\infty,\infty] \setminus \{0\}) \times W: \\
&\;\;\;\;\;\;\;\;\;\;\;\;\;\;\;\;\;\;\;\;\;\;\;\;\;\;\;\;\;\;\;\;\;\;\;\;\;\;\;\;\;\;
\int_{Q_y}\,\sum_{u \in H} \delta_{\left(Uy+V\lambda,\,\psi_H(x, v, u)\right)} \,d\lambda\in \cdot\Big\}\Big)
\end{align*}
and satisfying $\ka^q_\ast(\mathbb{M}^q \setminus B(\O,\varepsilon))<\infty$ for all $\varepsilon>0$.
\end{Theorem}

\begin{proof}[\textbf{Proof}]
Since this proof is similar to the proof of Theorem~\ref{thm:main:proc:diss} above with ingredients from \cite{roy:2010a}, we shall only sketch the main steps. For example,  it can be verified that $\ka^q_\ast\in \mathbf{M}_0(\mathbb{M}^q)$ using the same approach used in the proof of \Cref{lemma:diss:1}.

As before, fix Lipschitz functions $g_1,\,g_2 \in C^+_K(\mathbb{E}^q)$ and  $\epsilon_1,\,\epsilon_2>0$. For all $s \in \mathbb{Z}^d$ and $n \geq 1$, define $C_{s,n}:= [-n\mathbf{1}_d,n\mathbf{1}_d]\cap (s+K)$. With the help of this notation, $\Lamb_n^q$ can be rewritten as
$
\Lamb_n^q:=n^{p-d} \sum_{s \in H_n} \sum_{t \in C_{s,n}} \delta_{(n^{-1}t,\,\norm_n^{-1}\widetilde{X}^q_s)}.
$
Using the heuristics given before the proof of Theorem \ref{thm:main:proc:diss}, one can guess that the large deviation of $\Lamb_n^q$ would be same as that of
$$
\widehat{\Lamb}_n^q:=n^{p-d} \sum_{i=1}^\infty \sum_{s \in H_n} \sum_{t \in C_{s,n}} \delta_{(n^{-1}t, \norm_n^{-1}\psi_H(j_i,v_i,u_i\oplus s))}.
$$
Keeping this in mind, we define
\begin{eqnarray*}
\widehat{\ka}_n^q:=\frac{\norm_n^\alpha}{n^p} \bbP(\widehat{\Lamb}_n^q \in \cdot) \in \mathbf{M}_0(\mathbb{M}^q)
\end{eqnarray*}
and follow the proof of Lemma~\ref{diff:of:m_F:and:mhat_F} to establish that
\begin{equation}
\lim_{n \to \infty} \big|\ka_n^q(F_{g_1, g_2, \epsilon_1, \epsilon_2}) - \widehat{\ka}_n^q(F_{g_1, g_2, \epsilon_1, \epsilon_2})\big| = 0, \label{diff:ka_and_kahat}
\end{equation}
where $F_{g_1, g_2, \epsilon_1, \epsilon_2}$ is as in \eqref{eq:F}.

Moreover, we define for all $q\geq 0$,
\begin{eqnarray}
&&\hspace*{-0.6cm}\widetilde{\ka}_n^q(F_{g_1, g_2, \epsilon_1, \epsilon_2}) \nonumber\\
&&\hspace*{-0.3cm}:=\frac{\norm_n^\alpha}{n^p} \bbE\bigg[\sum_{i=1}^\infty \Big\{\Big(1 - \e^{-( n^{p-d}\sum_{s \in H_n} \sum_{t \in C_{s,n}}g_1(n^{-1}t,\,\norm_n^{-1} \psi_H(j_i, v_i, u_i \oplus s))-\epsilon_1)_+}\Big) \nonumber\\
&&\hspace*{-0.3cm} \;\;\;\;\;\;\;\;\;\;\;\;\;\;\;\;\;\;\;\;\;\;\times \Big(1 - \e^{-(n^{p-d}\sum_{s \in H_n} \sum_{t \in C_{s,n}}g_2(n^{-1}t,\,\norm_n^{-1} \psi_H(j_i, v_i, u_i \oplus s))-\epsilon_2)_+}\Big)\Big\}\bigg].\nonumber
\end{eqnarray}
Assuming that $g_1(t,y)=g_2(t,y)=0$ for all $y \in (-\eta,\eta)^{[-q\mathbf{1}_d, q\mathbf{1}_d]}$, and using \eqref{rate_of_growth:H_n} and an argument parallel to the one used in establishing \eqref{bound:on:Prob:of:B_n:compliment} above, it follows that
\[
\frac{\norm_n^\alpha}{n^p} \bbP\Big(\mbox{for more than one }i,\,\sum_{s\in H_n} \delta_{\norm_n^{-1} \psi_H(j_i, v_i, u_i \oplus s)}(A_\eta) \geq 1\Big) \to 0,
\]
from which we can establish a version of Lemma \ref{lemma:diff:of:mhat:and:mtilde} in this set up and conclude
\begin{equation}
    \lim_{n\to\infty}|\widehat{\ka}^q_n(F_{g_1,g_2,\epsilon_1,\epsilon_2})-\widetilde{\ka}^q_n(F_{g_1,g_2,\epsilon_1,\epsilon_2})| = 0. \label{diff:katilde_and_kahat}
\end{equation}

In light of \cref{propn:suff:condn:HLS:conv}, \eqref{diff:ka_and_kahat}, and \eqref{diff:katilde_and_kahat}, it is enough to prove that for all $q\geq 0$,
\begin{eqnarray} \label{conv:tildem_n:cons}
    \lim_{n\to\infty}\widetilde{\ka}_n^q(F_{g_1, g_2, \epsilon_1, \epsilon_2})= {\ka}_*^q(F_{g_1, g_2, \epsilon_1, \epsilon_2}).
\end{eqnarray}
We shall start with the special case when $h$ is supported on $W \times H_T$ for some $T \geq 1$. For such a function $h$, we have
\begin{align*}
&\widetilde{\ka}_{n}^q(F_{g_1, g_2, \epsilon_1, \epsilon_2})\\
&=\frac{1}{n^p}\int_{|x|>0}\int_W \sum_{u \in H_{n+T+q}} \Big\{\Big(1 - \e^{-(n^{p-d}\sum_{s \in H_{n}} \sum_{t \in C_{s,n}}g_1(n^{-1}t,\,\psi_H(x, v, s\ominus u))-\epsilon_1)_+}\Big)\\
& \;\;\;\;\;\;\;\;\;\;\;\;\;\;\;\;\;\;\;\;\;\;\;\;\;\;\;\;\;\;\times \Big(1 - \e^{-(n^{p-d}\sum_{s \in H_n} \sum_{t \in C_{s,n}}g_2(n^{-1}t,\,\psi_H(x, v, s\ominus u))-\epsilon_2)_+}\Big)\Big\}\\
&\hspace{3.9in} \nu(dv) \nu_\alpha(dx),
\end{align*}
from which, applying Lemma ~5.1 in \cite{roy:2010a}, \eqref{rate_of_growth:H_n} above and the fact that $g_1$ and $g_2$ are Lipschitz, it follows that
\begin{eqnarray*}
&&\hspace*{-0.6cm}\widetilde{\ka}_{n}^q(F_{g_1, g_2, \epsilon_1, \epsilon_2})\label{claim}\\
&&\hspace*{-0.6cm}=\frac{1}{n^p}\int_{|x|>0}\int_W \sum_{u \in H_{n+T+q}} \Big\{\Big(1 - \e^{-(n^{p-d}\sum_{z \in B_{u,n}} \sum_{t \in C_{u,n}}g_1(n^{-1}t,\,\psi_H(x, v, z))-\epsilon_1)_+}\Big) \nonumber\\
&& \;\;\;\;\;\;\;\;\;\;\;\;\;\;\;\;\;\;\;\;\;\;\;\;\;\;\;\;\times \Big(1 - \e^{-(n^{p-d}\sum_{z \in B_{u,n}} \sum_{t \in C_{u,n}}g_2(n^{-1}t,\,\psi_H(x, v, z))-\epsilon_2)_+}\Big)\Big\}\nonumber\\
&&\hspace{3.2in} \nu(dv) \nu_\alpha(dx)+o(1), \nonumber
\end{eqnarray*}
where $B_{u,n}:=\{z \in H_{T+q}: z\oplus u \in H_n\}$. The above equality and an argument similar to the one used in establishing (5.17) of \cite{roy:2010a} yield
\begin{align*}
&\lim_{n \to \infty}\widetilde{\ka}_{n}^q(F_{g_1, g_2, \epsilon_1, \epsilon_2})\\
&= l\int_{|x|>0}\int_W \int_\Delta \Big\{\Big(1 - \e^{-(\int_{Q_y}\sum_{z \in H_{T+q}} g_1(Uy+V\lambda,\,\psi_H(x, v, z))d\lambda-\epsilon_1)_+}\Big)\\
& \;\;\;\;\;\;\;\;\;\;\;\;\;\;\;\;\;\times \Big(1 - \e^{-(\int_{Q_y}\sum_{z \in H_{T+q}} g_2(Uy+V\lambda,\,\psi_H(x, v, z))d\lambda-\epsilon_2)_+}\Big)\Big\}
\,dy \,\nu(dv)\,\nu_\alpha(dx).
\end{align*}
This establishes \eqref{conv:tildem_n:cons} for $h$ with support $W\times H_T$ for some $T \geq 1$. The proof of \eqref{conv:tildem_n:cons} in the general case follows easily from the above by using a standard converging together technique (see the proofs of (5.21) and (5.22) in \cite{roy:2010a}) based on the inequalities used to establish Lemma~\ref{diff:of:m_F:and:mhat_F}. This completes the proof of Theorem~\ref{thm:main:proc:cons}.
\end{proof}

\begin{remark}
It is possible to interpret \cref{thm:main:proc:diss} as a special
case of \cref{thm:main:proc:cons} by setting  $p=d$, $l=1$, $\Delta=[-1,1]^d$, $H=\bbz^d,$ $U=I_d$ (the identity matrix of order $d$), $V=0$ along with the convention that $\mathbb{R}^{0}=\{0\}$ so that $Q_y=\{0\}$ for all $y \in [-1,1]^d$ and $\lambda$ is interpreted as the counting measure on $\{0\}$ (think of it as the zero-dimensional Lebesgue measure). However, since the above proof does not honour these conventions, a separate proof had to be given for \cref{thm:main:proc:diss}. Same remark applies to the two parts of \cref{thm:large_deviation} below.
\end{remark}

\begin{example}
In order to understand Theorem~\ref{thm:main:proc:diss} and its notations, let us consider Example 6.1 in \cite{roy:2010a} and apply \cref{thm:main:proc:cons} on it. This means $d=2$, $S=\mathbb{R}$, $\mu$ is the Lebesgue measure and $\{\phi_{(t_1,t_2)}\}$ is a measure preserving conservative $\mathbb{Z}^2$-action on $\mathbb{R}$ defined by
$
\phi_{(t_1,t_2)}(x)=x+t_1-t_2.
$
Take any $f \in L^\alpha(\bbr,\mu)$ and define a stationary $S\alpha S$ random field
$\{X_{(t_1,t_2)}\}$ as \linebreak
$X_{(t_1,t_2)} \eqdef \int_{\mathbb{R}} f\big(\phi_{(t_1,t_2)}(x)\big)\,
M(dx)$, $t_1,t_2 \in \mathbb{Z}$,
where $M$ is an $S \alpha S$ random measure on $\mathbb{R}$ with
control measure $\mu$. This representation of
$\{X_{(t_1,t_2)}\}$ is of the form $(\ref{repn_integral_stationary})$ with $c_{(t_1,t_2)} \equiv
1$.

As computed in \cite{roy:2010a}, in this case, $K=\{(t_1,t_2)\in \mathbb{Z}^2:\,t_1=t_2\}$, $p=d-p=l=1$, $H=F=\{(u_1,0):\,u_1 \in \mathbb{Z}\}$, and
$U=(1,0)^T$, $V=(1,1)^T$ so that $\Delta = [-2,2]$ and for all $y \in [-2,2]$,
\[
Q_y=\left\{
\begin{array}{ll}
\,[-(1+y),1], &\;\;\;y \in \left[-2,0\right), \\
\,[-1,1-y], &\;\;\;y \in [0,2].
\end{array}
\right.
\]

There is a standard Borel space
$(W,\mathcal{W})$ with a $\sigma$-finite measure $\nu$ on it such
that \eqref{mixed_moving_avg_repn_of_X_r_r_in_H} holds for some $h \in L^{\alpha}(W\times H, \nu \otimes \zeta_H)$, where
$\zeta_H$ is the counting measure on $H$, and  $M^\prime$ is a $S\alpha
S$ random measure on $W \times H$ with control measure $\nu \otimes
\zeta_H$. Note that for $u, s \in H$ with $u=(u_1, 0)$ and $s=(s_1, 0)$, $u \oplus s=(u_1+s_1,0)$, and $\pi(w_1, w_2)=(w_1-w_2, 0)$. Therefore, in this example, $\psi_H(x,v, (u_1,0))=\{xh(v, (u_1-w_1+w_2,0))\}_{-q \leq w_1, w_2 \leq q}$.

It was shown in \cite{roy:2010a} that
$
n^{-1} \sum_{|t_1|,\,|t_2| \leq n} \delta_{(4n)^{-1/\alpha}X_{(t_1,t_2)}}
$
converges weakly to a random element in the space of all Radon measures on \linebreak $[-\infty,\infty]\setminus\{0\}$. We take a sequence $\norm_n$ satisfying $n^{1/\alpha}/\norm_n \to 0$ and apply Theorem \ref{thm:main:proc:cons} to conclude that the following HLS convergence holds in $\mathbf{M}_0(\mathbb{M}^q)$:
\begin{align*}
&\frac{\norm_n^\alpha}{n}\bbP(\Lamb^q_n \in \cdot) \rightarrow \mu|_{[-2,2]} \otimes \nu_\alpha \otimes \nu \Big(\Big\{(y,x,v) \in [-2,2] \times ([-\infty,\infty] \setminus \{0\}) \times W: \\
&\;\;\;\;\;\;\;\;\;\;\;\;\;\;\;\;\;\;\;\;\;\;\;\;\;\;\;\;\;\;\;\;\;\;\;\;\;\;\;\;\;\;\;\;\;\;\;\;\;\;\;\;\;\;\;\;\;\;
\int_{Q_y}\,\sum_{u_1 \in \mathbb{Z}} \delta_{\left((y+\lambda,\lambda),\,\psi_H(x,v,(u_1,0))\right)} \,d\lambda\in \cdot\Big\}\Big),
\end{align*}
where
$
\Lamb^q_n=n^{-1}\sum_{|t_1|,\,|t_2| \leq n} \delta_{\left(n^{-1}(t_1, t_2), \; \norm_n^{-1}\{X_{(t_1-w_1,t_2-w_2)}\}_{-q \leq w_1, w_2 \leq q}\right)}
$.
\end{example}



The following corollary is a direct consequence of \cref{thm:main:proc:cons}. Its proof is very similar to that of \cref{Corollary:Order Statistics} and hence is skipped.

\begin{cor} \label{corollary:4.1}
Let $y>0$. Then as $n\to\infty$,
\begin{eqnarray*}
    &&\frac{\norm_n^{\alpha}}{n^p}\mathbb{P}\left(\max_{\|t\|_{\infty}\leq n}X_t>\norm_ny\right)\\
    &&\quad \to l\text{Leb}(\Delta)y^{-\alpha}
        \int_W(\sup_{u\in H} h^+(v,u))^{\alpha}+(\sup_{u\in H} h^-(v,u))^{\alpha} \nu(dv). 
\end{eqnarray*}
In particular, with $\tau^a_n$ as defined in \cref{Corollary:Order Statistics},
\begin{eqnarray*}
    &&\lim_{n\to\infty}\frac{\norm_n^{\alpha}}{n^p}\mathbb{P}(\tau^a_n\leq \lambda n)\\
    &&\;\;\;\;\;\;\;\;=\lambda ^pa^{-\alpha}l\mbox{Leb}(\Delta)\int_W(\sup_{u\in H} h^+(v,u))^{\alpha}+(\sup_{u\in H} h^-(v,u))^{\alpha} \nu(dv).
\end{eqnarray*}
\end{cor}

\section{Large deviation of the partial sum} \label{section:classical:large deviation}

In this section, we use our point process large deviation results to investigate the classical large deviation behaviour for the partial sum sequence of stationary symmetric stable random fields. As before, we consider two cases depending on whether the underlying group action is dissipative or conservative. To fix the notations, let $\{X_t\}_{t\in\mathbb{Z}^d}$ be a stationary symmetric $\alpha$-stable random field as before
and define the partial sum sequence
\begin{eqnarray} \label{partial sum}
    S_n=\sum_{\|t\|_{\infty}\leq n}X_t, \quad n\in\bbn.
\end{eqnarray}

Using continuous mapping arguments from the results of
 Theorem~3.1 and Theorem~4.1, respectively in \cite{roy:2010a}, one can establish the following weak convergence results. If $\{X_t\}_{t\in\mathbb{Z}^d}$  is generated by a dissipative action as in \cref{thm:main:proc:diss} having representation \eqref{repn_Possion_integral_X_n} with kernel function $f\in L^{\alpha}(W\times \bbz^d,\nu\otimes\zeta)$
satisfying
\begin{eqnarray}    \label{assump:1}
    \int_W \left(\sum_{u\in\bbz^d} \left|f(v,u)\right|\right)^{\alpha}\nu({\rm d}v)<\infty,
\end{eqnarray}
then
$
n^{-d/\alpha}S_n\Rightarrow C_f Z_\alpha
$,
where $Z_\alpha \sim S\alpha S(1)$ and
\begin{eqnarray}
C_f^\alpha:=2^d\int_W\left(\left(\sum_{u\in\bbz^d}f(v,u)\right)^+\right)^{\alpha}
+\left(\left(\sum_{u\in\bbz^d} f(v,u)\right)^-\right)^{\alpha}\, \nu({\rm d}v). \label{def:C_f}
\end{eqnarray}
On the other hand, if $\{X_t\}_{t\in\mathbb{Z}^d}$ is generated by a conservative action
as in \cref{thm:main:proc:cons} with $h\in L^{\alpha}(W\times H,\nu\otimes\zeta_H)$
satisfying
\begin{eqnarray}    \label{assump:2}
    \int_W \left(\sum_{u\in H} \left|h(v,u)\right|\right)^{\alpha}\nu({\rm d}v)<\infty,
\end{eqnarray}
then
$
\displaystyle n^{p-d-p/\alpha}S_n\weak  C_{l,\mathcal{V},h}Z_\alpha
$,
where
\begin{align}
     C_{l,\mathcal{V},h}^\alpha&:=l\left(\int_\Delta (\mathcal{V}(y))^{\alpha}\,{\rm d}y\right) \times \nonumber\\
     & \;\;\;\;\;\;\;\;\;\;\int_W\left(\left(\sum_{u\in H}h(v,u)\right)^+\right)^{\alpha}
    +\left(\left(\sum_{u\in H} h(v,u)\right)^-\right)^{\alpha}\, \nu({\rm d}v). \label{def:C_lVh}
\end{align}
We do not present the proofs of the above statements because they will also follow from our large deviation results; see \cref{thm:large_deviation} and Remark~\ref{remark:A} below. Note that the normalization for weak convergence of partial maxima and partial sum sequences are the same in the dissipative case but not in the conservative case. This is because the longer memory results in huge clusters and this causes the partial sum to grow faster than the maxima.

The following theorem deals with the classical large deviation issue of the partial sum sequence $S_n$ under
 the assumptions of Theorems ~\ref{thm:main:proc:diss} and \ref{thm:main:proc:cons}, respectively.
 The convergence used in these results is as in \cite{hult:lindskog:2006} with the space $\mathbf{S} = \bbr$
 and the deleted point $s_0 =0$, i.e. $\mathbf{S}_0=\bbr\backslash\{0\}$. This results in the space $\mathbb{M}_0(\bbr)$ of all Borel measures
 on $\bbr \backslash \{0\}$ that are finite outside any neighbourhood of $0$.
 The convergence in $\mathbb{M}_0(\bbr)$ implies vague convergence
 in $\bbr \backslash \{0\}$; see Lemma 2.1. in \cite{Lindskog:Resnick:Roy}. 

\begin{Theorem} \label{thm:large_deviation}
Let $\{X_t\}_{t\in\mathbb{Z}^d}$ be a stationary symmetric $\alpha$-stable random field
and $S_n$ be the partial sum sequence as defined in \eqref{partial sum}. Then the following large deviation results hold. \\
\noindent \textsl{(a)} \, If $\{X_t\}_{t\in\mathbb{Z}^d}$  is generated by a dissipative group action
 as in \cref{thm:main:proc:diss} having representation \eqref{repn_Possion_integral_X_n} with kernel function
 $f\in L^{\alpha}(W\times \bbz^d,\nu\otimes\zeta)$
satisfying \eqref{assump:1} and  $\{\gamma_n\}$ satisfying \eqref{gamma}, then 
\begin{eqnarray*}
    \frac{\gamma_n^{\alpha}}{n^d}\bbP(\gamma_n^{-1}S_n\in \cdot)
    \to C_f^\alpha\nu_{\alpha}(\cdot)\quad \mbox{ as } n\to\infty \mbox{ in }\mathbb{M}_0(\bbr),
\end{eqnarray*}
where $C_f$ is as in \eqref{def:C_f} and $\nu_\alpha$ is as in \eqref{defn:nu_alpha}. \\
\noindent \textsl{(b)} If $\{X_t\}_{t\in\mathbb{Z}^d}$ is generated by a conservative action
as in \cref{thm:main:proc:cons} with
 $h\in L^{\alpha}(W\times H,\nu\otimes\zeta_H)$
satisfying \eqref{assump:2}
and $\{\norm_n\}$ satisfying \eqref{beta}, then
\begin{eqnarray*}
    \mu_n(\cdot):=\frac{\norm_n^{\alpha}}{n^p}\bbP(n^{p-d}\norm_n^{-1}S_n\in \cdot)
    \to \mu(\cdot) \quad \mbox{ as } n\to\infty \mbox{ in }\mathbb{M}_0(\bbr),
\end{eqnarray*}
where
    $\mu(\cdot)= C_{l,\mathcal{V},h}^\alpha\nu_{\alpha}(\cdot)$
with $C_{l,\mathcal{V},h}$ as in \eqref{def:C_lVh} and  $\nu_\alpha$ as in \eqref{defn:nu_alpha}.
\end{Theorem}

\noindent The proof of this theorem is presented in the next subsection. For the point process large deviation result, we gave the detailed proof of the dissipative case and sketched the proof in the conservative case. In this case, we shall present the detailed proof of this theorem when the underlying action is conservative. The other case will follow similarly.

\begin{Remark} \label{remark:A}
(a) Let $\{X_t\}_{t\in\bbz^d}$ be an S$\alpha$S process. Then
    $S_n$ defined by \eqref{partial sum}
is an S$\alpha$S random variable as well.  We denote its scaling parameter
by $\sigma_n$. This means $S_n\stackrel{\mbox{\tiny d}}{=}\sigma_n Z_\alpha$ with $Z_\alpha\sim$S$\alpha$S$(1)$. If $\{\gamma_n\},\{c_n\}$ are sequences
of positive constants satisfying $n^{\ka}/\gamma_n\to 0$ for some $\ka>0$, then the following equivalences
hold for $C>0$:
\begin{itemize}
    \item[(i)] ${\displaystyle \frac{\gamma_n}{n^{\ka}c_n}\sigma_n\to C}$ as $n\to\infty$.
    \item[(ii)] ${\displaystyle \frac{\gamma_n^{\alpha}}{n^{\alpha\ka}}\mathbb{P}(c_n^{-1}S_n\in\cdot)\to C^\alpha\nu_\alpha(\cdot)}$ as $n\to\infty$ in $\mathbb{M}_0(\bbr)$.
    \item[(iii)] ${\displaystyle \frac{\gamma_n}{n^{\ka}c_n}S_n\Rightarrow CZ_\alpha}$ as $n\to\infty$.
\end{itemize}
Consequently, the large deviation behaviors in \Cref{thm:large_deviation}  imply the weak convergence results presented in the beginning of this section, and vice versa.

(b) If $\alpha\in\left(0,1\right]$ and $f\in L^{\alpha}(W\times \bbz,\nu\otimes\zeta)$
then assumption \eqref{assump:1} is satisfied. However, for $\alpha\in\left(1,2\right)$
this is unfortunately not necessarily the case. To see this, let $\{X_t\}_{t\in\bbz}$
be a moving average process of the form
$
    X_t=\sum_{j=-\infty}^t \beta_{t-j}Z_j
$,  $t\in\bbz$,
where $(Z_j)_{j\in\bbz}$ is an iid sequence following an S$\alpha$S(1) distribution with $\alpha>1$ and $\beta_j=j^{-\norm}$, $j\in\bbn$, for some $\alpha^{-1}<\norm<1$.
Clearly, \eqref{assump:1} is not satisfied since $\sum_{j=0}^\infty |\beta_j|=\infty$. Theorem~1 in \cite{Astrauskas:1983a} says that
$n^{-1/\alpha-1+\norm}S_n\Rightarrow CZ_{\alpha}$ as $n\to\infty$ for some $C >0$.
Hence, $\sigma_n\sim Cn^{1/\alpha+1-\gamma}$.
A conclusion of the equivalences in (a) is that for any sequence $\{\gamma_n\}$ with $n^{1/\alpha+(1-\norm)}/\gamma_n\to 0$ as $n\to\infty$,
\begin{eqnarray} \label{ex:scaling}
    \frac{\gamma_n^{\alpha}}{n^{1+(1-\norm)\alpha}}\mathbb{P}(\gamma_n^{-1}S_n\in\cdot)\to C^\alpha\nu_\alpha(\cdot) \quad \mbox{ as } n\to\infty \mbox{ in }\mathbb{M}_0(\bbr).
\end{eqnarray}
We see that the scaling in the large deviation behavior  in \cref{thm:large_deviation}~(a) under
assumption~\eqref{assump:1} differs from the scaling in \eqref{ex:scaling}.
Further examples for moving average processes with $\sum_{j=0}^{\infty}|\beta_j|=\infty$ whose scaling
$\sigma_n$ satisfies $n^{-1/\alpha}\sigma_n\to\infty$ can be found in \cite{whitt:2002}, \cite{Astrauskas:1983b,Astrauskas:1983a} and
\cite{Hsing:1999}.
\end{Remark}

\subsection{Proof of \cref{thm:large_deviation}}
As discussed earlier, we will prove this theorem only for the conservative case \textsl{(b)}. The dissipative case \textsl{(a)} can be dealt with in a similar fashion.

We shall first prove \cref{thm:large_deviation}~(b) for $h$ supported on $W\times H_T$ for some $T \geq 1$, and then use a converging together argument. To this end, for all $T \in \mathbb{N}$, set $h_T=h\1_{W\times H_T}$ and define $X_t^{(T)}$,
$\mu_{n,T}$, $\mu_T$ and $C_{l,\mathcal{V},h_{T}}$ by replacing $h$ by $h_T$ in the definition of
$X_t$,
$\mu_{n}$, $\mu$ and $C_{l,\mathcal{V},h}$, respectively.

\begin{Lemma} \label{Proposition:large_deviation:conservative}
Let $S_n^{(T)}=\sum_{\|t\|_\infty\leq n}X_t^{(T)}$, $n\in\bbn$. Then
\begin{eqnarray*}
    \mu_n^{(T)}(\cdot):=\frac{\norm_n^{\alpha}}{n^p}\bbP(n^{p-d}\norm_n^{-1}S_n^{(T)}\in \cdot)
    \to \mu^{(T)}(\cdot) \quad \mbox{ as } n\to\infty \mbox{ in }\,\mathbb{M}_0(\bbr).
\end{eqnarray*}
\end{Lemma}
\begin{proof}[\textbf{Proof}]
Since the proof is very similarly to the proof of Theorem~6.1
in \linebreak \cite{hult:samorodnitsky:2010} we will give only a short sketch.
 The idea is that for any $0<\epsilon<1$,  $S_n$ is divided into three parts
\begin{eqnarray*}
    S_n^{(T)}&=&\sum_{\|t\|_\infty\leq n}X_t^{(T)}\left[\1_{\{|X_t^{(T)}|\leq \epsilon\}}+\1_{\{\epsilon<|X_t^{(T)}|\leq \epsilon^{-1}\}}
        +\1_{\{|X_t^{(T)}|> \epsilon^{-1}\}}\right]\\
        &=:&S_n^{(1)}+S_n^{(2)}+S_n^{(3)}.
\end{eqnarray*}
In the following we investigate the second term. Define \linebreak $g_\epsilon:[-1,1]^d\times[-\infty,\infty]\backslash\{0\}\to\bbr$
with $g_\epsilon(t,x)=x\1_{\{\epsilon<|x|\leq \epsilon^{-1}\}}$.  Since
\begin{eqnarray*}
    &&\hspace*{-0.6cm}\kappa_{*,T}^{0}(\xi\in \mathbb{M}^0:\xi([-1,1]^d\times\{|x|=\epsilon\mbox{ or }\epsilon^{-1}\})>0)\\
        &&\hspace*{-0.3cm}\leq l Leb(\Delta)\sum_{u\in H_T}\nu_\alpha\otimes\nu\left(\left\{(x,v)\in[-\infty,\infty]\backslash\{0\}\times W:
        |xh(v,u)|=\epsilon \mbox{ or } \epsilon^{-1}\right\}\right)\\
        &&\hspace*{-0.3cm}=0,
\end{eqnarray*}
the continuous-mapping theorem (see Lemma A.2 in \cite{hult:samorodnitsky:2010}) and \cref{thm:main:proc:cons} give
\begin{eqnarray*}
    \lefteqn{\frac{\gamma_n^\alpha}{n^p}\mathbb{P}(n^{p-d}S_n^{(2)}\in\cdot)=\frac{\gamma_n^\alpha}{n^p}\mathbb{P}(g_\epsilon(\Lambda_{T,n}^{0})\in\cdot)}\\
    &&\to lLeb|_\Delta\otimes\nu_\alpha\otimes\nu\bigg(\bigg\{(y,x,v)\in\Delta\times [-\infty,\infty]\backslash\{0\}\times W:\\
    &&\hspace*{4.5cm}\mathcal{V}(y)\sum_{u\in H_T}xh(v,u)\1_{\{\epsilon<|xh(v,u)|\leq \epsilon^{-1}\}}\in\cdot\bigg\}\bigg)\\
    &&=:\mu^{(T)}_\epsilon(\cdot)
\end{eqnarray*}
as $n\to\infty$ in $\mathbb{M}_0(\bbr)$.  Moreover, for any bounded continuous map $g: \bbr \to \bbr$ that vanishes in a neighbourhood of $0$, say $(-\eta,\eta)$ for some $\eta>0$,
by dominated convergence  the limit
\begin{eqnarray*}
    \mu^{(T)}_\epsilon(g)\hspace*{-0.3cm}&=&\hspace*{-0.4cm} \int_\Delta \int_{\bbr\backslash\{0\}}\int_Wg\left(\mathcal{V}(y)\sum_{u\in H_T}xh(v,u)\1_{\{\epsilon<|xh(v,u)|\leq \epsilon^{-1}\}}\right) \nu(dv)\nu_\alpha(dx)dy\\
    &\to&\hspace*{-0.4cm}\int_\Delta \int_{\bbr\backslash\{0\}}\int_Wg\left(\mathcal{V}(y)\sum_{u\in H_T}xh(v,u)\right)\, \nu(dv)\,\nu_\alpha(dx)\,dy= \mu^{(T)}(g)
\end{eqnarray*}
holds as $\epsilon \to 0$.  Dominated convergence theorem can be applied in the above limit since $\mathcal{V}$ is bounded (Lemma~5.1 in \cite{roy:2010a}) and  we assume \eqref{assump:2}.

Finally, if we  show that for any $\delta>0$
\begin{eqnarray} \label{eq: 6.1}
    \lim_{\epsilon\downarrow 0}\limsup_{n\to\infty}
    \frac{\norm_n^{\alpha}}{n^p}\bbP\left(\left|\sum_{\|t\|_{\infty}\leq n}X_t^{(T)}\1_{\{|X_t^{(T)}|\leq\norm_n\epsilon\}}\right|>\norm_n n^{d-p}\delta\right)=0,
\end{eqnarray}
then Lemma~\ref{Proposition:large_deviation:conservative} will follow step by step as in the proof of Theorem~6.1
in \cite{hult:samorodnitsky:2010} by a converging together argument.

To prove \eqref{eq: 6.1}, note that 
\begin{eqnarray*}
    \lefteqn{\frac{\norm_n^{\alpha}}{n^p}\bbP\left(\left|\sum_{\|t\|_{\infty}\leq n}X_t^{(T)}\1_{\{|X_t^{(T)}|\leq \norm_n\epsilon\}}\right|>\norm_nn^{d-p}\delta\right)}\\
    &&=\frac{\norm_n^{\alpha}}{n^p}\bbP\left(\left|\sum_{s\in H_n}m(s,n)X_s^{(T)}\1_{\{|X_s^{(T)}|\leq \norm_n\epsilon\}}\right|>\norm_nn^{d-p}\delta\right)
\end{eqnarray*}
with $m(s,n):=|[-n\mathbf{1}_d,-n\mathbf{1}_d]\cap(s+K)|$ for $s\in H$.
First, we would like to point out that if $N(u\oplus s)\leq T$ for some $u\in H$, then
it can easily be shown that $N(s)-T\leq N(u)\leq N(s)+T$. Hence, we have the
representation
\begin{eqnarray*}
    X_s^{(T)} = \int_{W \times H_{N(s)+T}\cap H^c_{N(s)-T}} h_T(v,u \oplus s)\,M^\prime(dv,du).
\end{eqnarray*}
From this we see that
for $s_1,s_2\in H$ with $N(s_1)+T\leq N(s_2)-T$ the intersection
$H_{N(s_2)-T}^c\cap H_{N(s_1)+T}$ is empty so that $X_{s_1}^{(T)}$ and $X_{s_2}^{(T)}$ are independent.

Let $s_1,s_2\in H$ and $u\in H$ with $N(u\oplus s_1)\leq T$. Then
\begin{eqnarray} \label{ineq:N1}
    N(u\oplus s_2)\geq N(s_2\ominus s_1)- N(u \oplus s_1)\geq N(s_2\ominus s_1)-T. 
\end{eqnarray}
We define the positive finite constant
\begin{eqnarray*}
    c:=\min\{\|Ui+V\norm\|_{\infty}:i\in\bbz^p\backslash\{\mathbf{0}_p\},\norm\in\bbr^q\},
\end{eqnarray*}
 and $c^*:=\inf\{z\in\bbn: 1/c\leq z\}=\lceil c^{-1}\rceil $.
If  $s_1:=x_k+U(c^*(2T+1)i_1+ j)\in H$ and $s_2:=x_k+U(c^*(2T+1)i_2+ j)\in H$ for some $i_1,i_2,j\in\bbz^p$, $i_1\not=i_2$, then
\begin{eqnarray} \label{ineq:N2}
    N(s_2\ominus s_1)&=&\min\{\|s_2-s_1+v\|_{\infty}:v\in K\} \nonumber\\
        &\geq &(2T+1)c^*\min\{\|Ui+V\norm\|_{\infty}:i\in\bbz^p\backslash\{\mathbf{0}_p\},\,\norm\in\bbr^q\}\nonumber\\
        &= &(2T+1).
\end{eqnarray}
A conclusion of \eqref{ineq:N1} and \eqref{ineq:N2} is that
    $N(u\oplus s_2)\geq T+1$
and finally, \linebreak $X_{s_1}^{(T)}=X_{x_k+U(c^*(2T+1)i_1+ j)}^{(T)}$ and $X_{s_2}^{(T)}=X_{x_k+U(c^*(2T+1)i_2+ j)}^{(T)}$
are independent. In the following, we assume without loss of generality that $n+L$ is a multiple of $c^*(2T+1)$ where
$L:=\max_{k=1,\ldots,l}\|x_k\|_{\infty}$
and define
$n':=(n+L)/(c^*(2T+1))$.
This gives $H_n \subseteq  [-n\mathbf{1}_d,n\mathbf{1}_d]$ and
\begin{align*}
     H_n 
        &\subseteq \bigcup_{k=1}^l\{x_k+U(c^*(2T+1)i+j): \, j\in [-c^*T\mathbf{1}_p,c^*T\mathbf{1}_p],\,i\in[- n'\mathbf{1}_p, n'\mathbf{1}_p]\}.
\end{align*}
We define $s_{k,i,j}:=x_k+U(c^*(2T+1)i+j)$ for $i,j\in\bbz^p$, $k\in\{1,\ldots,l\}$. Then
     $H_n\subseteq \{s_{k,i,j}: \, k\in\{1,\ldots,l\},\,j\in [-c^*T\mathbf{1}_p,c^*T\mathbf{1}_p],\,i\in[- n'\mathbf{1}_p, n'\mathbf{1}_p]\}.$
The independence of the sequence $(X_{s_{k,i,j}}^{(T)})_{i\in\bbz^p}$  for fixed $j\in\bbz^p$ and
$k\in\{1,\ldots,l\}$, Markov's inequality and
Karamata's Theorem  (cf.~\cite{resnick:2007}, eq.~(2.5) on p.~36) result in
\begin{eqnarray*}
   \lefteqn{\frac{\norm_n^{\alpha}}{n^p}\bbP\left(\left|\sum_{s\in H_n}m(s,n)X_s^{(T)}\1_{\{|X_s^{(T)}|\leq \norm_n\epsilon\}}\right|>\norm_nn^{d-p}\delta\right)}\\
    &&\leq {\rm const. }\,\norm_n^{\alpha}
            \bbP(|X_1^{(T)}|>\norm_n\epsilon)\epsilon^2 \frac{1}{n^p} \sum_{s\in H_n}\frac{m(s,n)^2}{n^{2(d-p)}}
            \leq {\rm const. }\,\epsilon^{2-\alpha}\stackrel{\epsilon\downarrow 0}{\to}0.
\end{eqnarray*}
In the last inequality, we used \eqref{rate_of_growth:H_n} and Lemma~5.1 in \cite{roy:2010a}, which says
that $m(s,n)/n^{(d-p)}$ is uniformly bounded.
\end{proof}

In order to complete the converging together argument and establish Theorem~\ref{thm:large_deviation} (b) from Lemma~\ref{Proposition:large_deviation:conservative}, we need one more lemma.

\begin{Lemma} \label{Lemma:partial_sum:conservative}
$S_n-S_n^{(T)}\sim S\alpha S(\sigma_{T,n})$ where
\begin{eqnarray*}
    \lim_{T\to\infty}\limsup_{n\to\infty}\frac{\sigma_{T,n}}{n^{\frac{p}{\alpha}+(d-p)}}=0.
\end{eqnarray*}
\end{Lemma}
\begin{proof}[\textbf{Proof}] By the decomposition
\begin{eqnarray*}
    S_n-S_n^{(T)}\hspace*{-0.2cm}&=&\hspace*{-0.2cm}\sum_{s\in H_n}m(s,n)[X_s-X_s^{(T)}]\\
        &&\hspace*{-2.3cm}=\left[\int_{W\times H_{n+T}}+\int_{W\times H_{n+T}^c}\right]\left(\sum_{s\in H_n}m(s,n)  h(v,u\oplus s)\1_{\{N(u\oplus s)>T\}}\right)\,M'({\rm d}v,{\rm d}u),\\
\end{eqnarray*}
the random variable $S_n-S_n^{(T)}$ is $S\alpha S$ with scale parameter
\begin{eqnarray} \label{sigma:2}
    \sigma_{T,n}=(\sigma_{1,T,n}^{\alpha}+\sigma_{2,T,n}^{\alpha})^{1/\alpha},
\end{eqnarray}
where
\begin{eqnarray*}
    \sigma_{1,T,n}^{\alpha}&=&\int_{W\times H_{n+T} }\left|\sum_{s\in H_n}
        m(s,n)h(v,u\oplus s)\right|^{\alpha}\1_{\{N(u\oplus s)>T\}}\zeta_H({\rm d}u) \nu({\rm d}v),\\
            \sigma_{2,T,n}^{\alpha}&=&\int_{W\times H_{n+T}^c}\left|\sum_{s\in H_n }
            m(s,n)h(v,u\oplus s)\1_{\{N(u\oplus s)>T\}}\right|^{\alpha}\zeta_H({\rm d}u) \nu({\rm d}v).
\end{eqnarray*}
In the following, we will use that there exists a constant $\ka_0$ such that \linebreak $m(s,n)/n^{(d-p)}\leq \ka_0$
for all $s\in H$ and $n\in\bbn$  (cf. \cite{roy:2010a}, Lemma~5.1) and $|H_{n+T}|\sim c (n+T)^p\sim cn^p$
(cf. \eqref{rate_of_growth:H_n}).
The first term in \eqref{sigma:2} has the representation
\begin{align} \label{eq:v1}
    \frac{\sigma_{1,T,n}^{\alpha}}{n^{p+\alpha(d-p)}}&=\frac{1}{n^p}\int_{W} \sum_{u\in H_{n+T}}\left|\sum_{s\in H_n }
                \frac{m(s,n)}{n^{d-p}}h(v,u\oplus s)\1_{\{N(u\oplus s)>T\}}\right|^{\alpha} \nu({\rm d}v) \nonumber\\
        &\hspace*{-0.9cm}\leq\mbox{const.}\frac{|H_{n+T}|}{n^p}\int_{W} \left(\sum_{j\in H_T^c}\left|h(v,j)\right|\right)^{\alpha} \nu({\rm d}v)\nonumber\\
        &\hspace*{-0.9cm}\stackrel{n\to\infty}{\longrightarrow}\mbox{const.}\int_{W} \left(\sum_{j\in H_T^c}\left|h(v,j)\right|\right)^{\alpha} \nu({\rm d}v)
            \stackrel{T\to\infty}{\longrightarrow}0
\end{align}
by dominated convergence and assumption \eqref{assump:2}.
It is easy to check that, if $\alpha\leq 1$, then
\begin{align*}
    \frac{\sigma_{2,T,n}^{\alpha}}{n^{p+\alpha(d-p)}}&\leq\frac{1}{n^{p}}\int_{W} \sum_{u\in H_{n+T}^c}\left|\sum_{s\in H_n} \frac{m(s,n)}{n^{(d-p)}} |h(v,u\oplus s)|\1_{\{N(u\oplus s)>T\}}\right|^{\alpha}\nu({\rm d}v) \\
        &\leq\frac{\mbox{const. }}{n^{p}}\int_{W}  \sum_{s\in H_n}\sum_{u\in H_{n+T}^c} |h(v,u\oplus s)|^{\alpha}\1_{\{N(u\oplus s)>T\}}\nu({\rm d}v)\\
        &\leq\mbox{const. }\int_{W} \sum_{j\in H_T^c} |h(v,j)|^{\alpha}\nu({\rm d}v)\stackrel{T\to\infty}{\longrightarrow}0,
\end{align*}
by dominated convergence and $h\in L^{\alpha}(W\times H,\nu\otimes\zeta_H)$. On the other hand, if $1<\alpha<2$, then
\begin{align*}
    &\frac{\sigma_{2,T,n}^{\alpha}}{n^{p+\alpha(d-p)}}\\
    &\leq\mbox{const. }\int_{W} \sum_{u\in H_{n+T}^c}\left(\sum_{j\in H_T^c}\left|h(v,j)\right|\right)^{\alpha}
        \frac{1}{n^p} \times\\
    &\hspace{2in}    \left(\frac{\sum_{s\in H_n} \left| h(v,u\oplus s)\1_{\{N(u\oplus s)>T\}}\right|}{\sum_{j\in H_T^c}\left|h(v,j)\right|}\right)^{\alpha}\nu({\rm d}v)\\
        &\leq\mbox{const. }\int_{W} \left(\sum_{j\in H_T^c}\left|h(v,j)\right|\right)^{\alpha}
                     \times\\
        & \hspace{1.7in}\frac{\sum_{u\in H_{n+T}^c}\sum_{s\in H_n}
                     \left|h(v,u\oplus s)\right|\1_{\{N(u\oplus s)>T\}}}{n^p\sum_{j\in H_T^c}\left|h(v,j)\right|}\nu({\rm d}v)\\
        &\leq\mbox{const. }\int_{W} \left(\sum_{j\in H_T^c}\left|h(v,j)\right|\right)^{\alpha}\nu({\rm d}v)\stackrel{T\to\infty}{\to}0,
\end{align*}
by dominated convergence and assumption \eqref{assump:2}.
To summarize
\begin{eqnarray} \label{eq:v2}
        \lim_{T\to\infty}\limsup_{n\to\infty}\frac{\sigma_{2,T,n}^{\alpha}}{n^{p+\alpha(d-p)}}=0.
\end{eqnarray}
A conclusion of \eqref{sigma:2}-\eqref{eq:v2} is that
$
    \lim_{T \to \infty}\limsup_{n\to\infty}\frac{\sigma_{T,n}^{\alpha}}{n^{p+\alpha(d-p)}}=0.
$
\end{proof}

Now we are ready to prove Theorem~\ref{thm:large_deviation}~(b). We have to show \linebreak $\lim_{n\to\infty}\mu_n(g)=\mu(g)$ for any bounded continuous map $g: \bbr \to \bbr$ that vanishes in a neighbourhood of $0$; see Theorem 2.1 in \cite{hult:lindskog:2006}.
As noted in the appendix of \cite{hult:samorodnitsky:2010}, p.33,
we can further assume that $g$ is a Lipschitz function. For such a function $g$ and any $\delta>0$, $|\mu(g)-\mu_n(g)|$ is bounded by
\begin{eqnarray} \label{eq:A1:2}
    &&\hspace*{-0.5cm}\; |\mu(g)-\mu^{(T)}(g)|+\left|\mu^{(T)}(g)-\bbE\left(\frac{\norm_n^{\alpha}}{n^p}g(n^{p-d}\norm_n^{-1}S_n^{(T)})\right)\right| \nonumber\\
        &&\hspace*{-0.5cm}\;\;\;+\left|\bbE\left(\left(\frac{\norm_n^{\alpha}}{n^p}g(n^{p-d}\norm_n^{-1}S_n^{(T)})-\frac{\norm_n^{\alpha}}{n^p}g(n^{p-d}\norm_n^{-1}S_n)\right)
        \1_{\{n^{p-d}\norm_n^{-1}|S_n-S_n^{(T)}|> \delta\}}\right)\right|\nonumber\\
   &&\hspace*{-0.5cm}\;\;\;+ \left|\bbE\left(\left(\frac{\norm_n^{\alpha}}{n^p}g(n^{p-d}\norm_n^{-1}S_n^{(T)})-\frac{\norm_n^{\alpha}}{n^p}g(n^{p-d}\norm_n^{-1}S_n)\right)
        \1_{\{n^{p-d}\norm_n^{-1}|S_n-S_n^{(T)}|\leq \delta\}}\right)\right|\nonumber\\
   &&\hspace*{-0.5cm}\;=:I_{T,n,1}+I_{T,n,2}+I_{T,n,3}+I_{T,n,4}. \nonumber
\end{eqnarray}
We shall show that $\lim_{T\to\infty}\limsup_{n\to\infty}I_{T,n,i}=0$ for $i=1,2,3$, which combined with $\lim_{\delta\downarrow 0}\lim_{T\to\infty}\limsup_{n\to\infty}I_{T,n,4}=0$ will prove this theorem.

First, using dominated convergence and assumption~\eqref{assump:2}, we obtain
for any Borel $B\subseteq \bbr\backslash\{0\}$,
\begin{eqnarray} \label{eq:v5:2}
\mu^{(T)}(B)
\stackrel{T\to\infty}{\to}
        \mu(B).
\end{eqnarray}
A consequence of Portmanteau-Theorem (Theorem 2.4 in \cite{hult:lindskog:2006}) is $\mu^{(T)}\to \mu$ as $T\to\infty$ in $\mathbb{M}_0(\bbr)$, and
    $\lim_{T\to\infty}\limsup_{n\to\infty}I_{T,n,1}=0.$
Moreover,  \Cref{Proposition:large_deviation:conservative} results in
    $\lim_{T\to\infty}\limsup_{n\to\infty}I_{T,n,2}=0.$

Next, for any $\delta>0$, we have
\begin{eqnarray*}
    I_{T,n,3}\leq \frac{\norm_n^{\alpha}}{n^p}2\|g\|_{\infty}\bbP(n^{p-d}\norm_n^{-1}|S_n-S_n^{(T)}|>\delta).
\end{eqnarray*}
Obviously, a conclusion of \Cref{Lemma:partial_sum:conservative} is that
$S_n-S_n^{(T)}\sim S\alpha S(\sigma_{T,n})$ with $\norm_n n^{d-p}\sigma_{T,n}^{-1}\to\infty$ if
$n^{p/\alpha}/\norm_n\to 0$ and hence
\begin{eqnarray*}
    \lim_{n\to\infty}\frac{\norm_n^{\alpha}}{n^p}\bbP(n^{p-d}\norm_n^{-1}|S_n-S_n^{(T)}|>\delta)
    =\lim_{n\to\infty} \frac{\sigma_{T,n}^\alpha}{n^{p+\alpha(d-p)}}\bbP(|Z_\alpha|>\delta)=0,
\end{eqnarray*}
where $Z_\alpha \sim S\alpha S(1)$ . Therefore,
  $  \lim_{T\to\infty}\limsup_{n\to\infty}I_{T,n,3}=0.$

Let $\eta>0$ such that $g(x)=0$ for $x\in(-\eta,\eta)$.
Suppose that $\delta<\eta/2$. If either $|g(n^{p-d}\norm_n^{-1}S_n)|>0$ or
$|g(n^{p-d}\norm_n^{-1}S_n^{(T)})|>0$, we have  $n^{p-d}\norm_n^{-1}|S_n^{(T)}|>\eta/2$
on $\{n^{p-d}\norm_n^{-1}|S_n-S_n^{(T)}|\leq \delta\}$. This results in
$$
    I_{T,n,4}\leq\sup_{|x-y|\leq \delta}|g(x)-g(y)|\frac{\norm_n^{\alpha}}{n^p}
                \bbP(n^{p-d}\norm_n^{-1}|S_n^{(T)}|>\eta/2).
$$
Using \Cref{Proposition:large_deviation:conservative}, \eqref{eq:v5:2} and the fact that $g$ is a Lipschitz function, it follows finally that
 $  \lim_{\delta\downarrow 0}\lim_{T\to\infty}\limsup_{n\to\infty}I_{T,n,4}=0$. This proves Theorem~\ref{thm:large_deviation} (b). \\

\noindent \textbf{Acknowledgement.} The authors would like to thank Gennady Samorodnitsky for some useful discussions.

\end{document}